\documentclass[final,times]{elsarticle}

\usepackage{hyperref}

\journal{arXiv.org}

\usepackage{todonotes}
\usepackage{tikz}
\usepackage{graphicx}
\usepackage{pstricks, pst-node, pst-plot, pstricks-add}

\usepackage[center]{caption}
\usepackage{subcaption}

\tikzstyle{nodino}=[circle,draw,fill,inner sep=0pt,minimum size=0.5mm]
\tikzstyle{infinito}=[circle,inner sep=0pt,minimum size=0mm]
\tikzstyle{nodo}=[circle,draw,fill,inner sep=0pt,minimum size=1mm]

\usepackage{listings}
\usepackage{color}

\definecolor{dkgreen}{rgb}{0,0.6,0}
\definecolor{gray}{rgb}{0.5,0.5,0.5}
\definecolor{mauve}{rgb}{0.58,0,0.82}

\usepackage{amsmath}

\usepackage{amsfonts}

\newcommand{\N}{\mathbb{N}}
\newcommand{\Q}{\mathbb{Q}}
\newcommand{\R}{\mathbb{R}}
\newcommand{\C}{\mathbb{C}}

\usepackage{amsopn}

\DeclareMathOperator*{\sech}{sech}
\DeclareMathOperator*{\cn}{cn}
\DeclareMathOperator*{\dn}{dn}
\DeclareMathOperator*{\arccn}{arc\,cn}

\DeclareMathOperator*{\lspan}{span}
\DeclareMathOperator*{\sgn}{sgn}

\newtheorem{theorem}{Theorem}[section]
\newtheorem{lemma}[theorem]{Lemma}
\newtheorem{proposition}[theorem]{Proposition}

\newdefinition{remark}[theorem]{Remark}

\newproof{proof}{Proof}

\begin{document}

\begin{frontmatter}

\title{Standing waves for the NLS on the double-bridge graph and a rational-irrational dichotomy}


\author[bicocca]{Diego Noja}
\ead{diego.noja@unimib.it}

\author[bicocca]{Sergio Rolando}
\ead{sergio.rolando@unimib.it}

\author[bicocca]{Simone Secchi}
\ead{simone.secchi@unimib.it}

\address[bicocca]{Dipartimento di Matematica e Applicazioni, Universit\`a
	di Milano Bicocca,  via R. Cozzi 55, 20125 Milano, Italy}

\begin{abstract}
We study a boundary value problem related to the search of standing waves  for the nonlinear Schr\"odinger equation (NLS) on graphs. 
Precisely we are interested in characterizing the standing waves of NLS posed on the  {\it double-bridge graph}, in which two semi-infinite half-lines are attached at a circle at different vertices. At the two vertices the so-called Kirchhoff boundary conditions are imposed. The configuration of the graph is characterized by two lengths, $L_1$ and $L_2$, and we are interested in the existence and properties of standing waves of given frequency $\omega$. For every $\omega>0$ only solutions supported on the circle exist (cnoidal solutions), and only for a rational value of $L_1/L_2$; they can be extended to every $\omega\in \R\ .$ We study, for $\omega<0$, the solutions periodic on the circle but with nontrivial components on the half-lines. The problem turns out to be equivalent to a nonlinear boundary value problem in which the boundary condition depends on the spectral parameter $\omega$. After classifying the solutions with rational $L_1/L_2$, we turn to $L_1/L_2$ irrational showing that there exist standing waves only in correspondence to a countable set of frequencies $\omega_n$. Moreover we show that the frequency sequence $\{\omega_n\}_{n \geq 1}$ has a cluster point at $-\infty$ and it admits at least a finite limit point, in general non-zero. Finally, any negative real number can be a limit point of a set of admitted frequencies up to the choice of a suitable irrational geometry $L_1/L_2$ for the graph. These results depend on basic properties of diophantine approximation of real numbers.
\end{abstract}


\end{frontmatter}


\section{Introduction and main results} \label{SEC:intro}
The analysis of nonlinear equations on graphs, especially nonlinear Schr\"odinger equation (NLS), is a new and rapidly growing research subject, which already produced a wealth of interesting results (for review and references see \cite{CFN17}). Roughly speaking a metric graph is a structure built by edges connected at vertices. Some of the edges may be of infinite length. On the edges a differential operator is given, with suitable boundary condition at vertices which makes it self-adjoint. This generates a dynamics (Wave, Heat, Schr\"odinger, Dirac or other).
The attractive feature of these mathematical models is the complexity allowed by the graph structure, joined with the one dimensional character of the equations. While they are certainly an oversimplification in many real problems coming from Physics in which geometry and transversal directions are not negligible, they however appear indicative of several dynamically interesting phenomena non typical or not expected in more standard frameworks. This is already true at the level of the linear Schr\"odinger equation, the so called Quantum Graphs theory, where an enormous literature exists (\cite{BerKu} and reference therein). 
The most studied topic in the context of nonlinear Schr\"odinger  equation is certainly existence and characterization of standing waves \cite{GW1,GW2,PS16}. More particularly several results are known about ground states (standing waves of minimal energy at fixed mass, i.e. $L^2$ norm) as regard existence, non existence and stability properties, depending on various characteristics of the graph \cite{acfn-aihp, AST1,AST2, AST3,CFN17}. A distinguished role is played by topology, by the vertex conditions and by the possible presence of external potentials.\\ In this paper we are interested in a special example which reveals an unsuspected fine structure of the set of standing waves when the metric properties of the graph are taken in account, and a relation with nonstandard boundary value problem and their solutions.
\\
Namely we consider a metric graph $\mathcal{G}$ made up of two half lines joined by
two bounded edges, i.e., a so-called \textit{double-bridge graph} (see Fig.\ref{db}). We may
also think at $\mathcal{G}$ as a ring with two half lines attached in two
distinct vertices. The half lines will be both identified with the interval 
$[0,+\infty )$, while the bounded edges will be represented by two bounded
intervals of lengths $L_1>0$ and $L_2 \geq L_1$, precisely 
$\left[0,L_1\right] $ and $[L_1,L]$ with $L=L_1+L_2 $. 
\vskip10pt

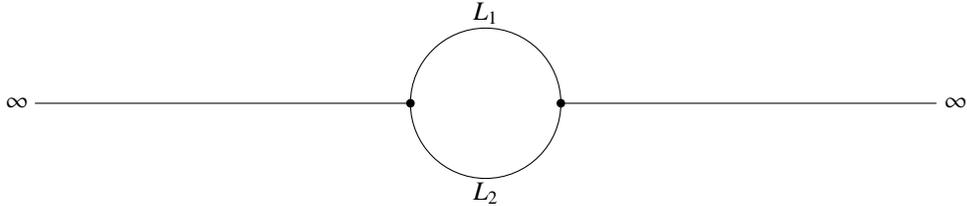
\begin{figure}[h]
	\begin{center}
		\begin{tikzpicture}
		\node at (-6,0)[infinito](1){};
		\node at (-1,0)[nodo](2){};
		\node at (1,0)[nodo](3){};
		\node at (6,0)[infinito](4){};
		\draw[-] (1)--(2);
		\draw[-] (2) arc(180:-180:1);
		\draw[-] (3)--(4);
		\node at (-6.23,0){$\infty$};
		\node at (6.25,0){$\infty$};
		\node at (0,1.2){$L_1$};
		\node at (0,-1.2){$L_2$};
		\end{tikzpicture}
		\caption{The double-bridge graph.}
		   \label{db}
	\end{center}
\end{figure}

\vskip10pt
A function $\psi $ on $\mathcal{G}$ is a cartesian product 
$\psi(x_1,...,x_4)=(\psi_1\left(x_1\right) ,\ldots,\psi_4\left(x_4\right))$ 
with $x_j\in I_j$ for $j=1,\ldots,4$, where $I_1=[0,L_1]$, $I_2=[L_1,L]$ and $I_3=I_4=[0,+\infty)$. \\
Then a Schr{\"{o}}dinger operator $H_{\mathcal{G}}$ on $\mathcal{G}$ is defined as 
\begin{equation}\label{operator}
H_{\mathcal{G}}\psi \left( x_1,\ldots,x_4\right) 
=\left( -\psi_1^{\prime\prime }\left( x_1\right) ,\ldots,-\psi_4^{\prime \prime }\left(x_4\right) \right),
\quad x_j\in I_j,
\end{equation}
with domain $D\left( H_{\mathcal{G}}\right) $ given by the functions $\psi $
on $\mathcal{G}$ whose components satisfy $\psi_j\in H^{2}(I_j)$ 
together with the so-called \textit{Kirchhoff boundary conditions}, i.e., 
\begin{equation}
\psi_1(0)=\psi_2(L)=\psi_3(0),\quad \psi_1(L_1)=\psi_2(L_1)=\psi_4(0),  
\label{continuity}
\end{equation}
\begin{equation}
\psi_1^{\prime}(0)-\psi_2^{\prime}(L)+\psi_3^{\prime}(0)
=\psi_1^{\prime}(L_1)-\psi_2^{\prime}(L_1)-\psi_4^{\prime}(0)=0.
\label{derivatives}
\end{equation}
As it is well known, the operator $H_{\mathcal{G}}$ is self-adjoint on the domain $D(H_{\mathcal{G}})$, and it generates a unitary Schr\"odinger dynamics. Essential information about its spectrum is given in Appendix.\\
We perturb this linear dynamics with a focusing cubic term, namely we consider the following nonlinear Schr{\"{o}}dinger equation on $\mathcal{G}$
\begin{equation}
i\frac{d \psi_t}{dt}=H_{\mathcal{G}}\psi_t-\left| \psi_t\right|^2\psi_t  
\label{NLS}
\end{equation}
where the nonlinear term $|\psi_t|^2\psi_t$ is a shortened notation for 
$(|\psi_{1,t}|^2\psi_{1,t},\ldots,|\psi_{4,t}|^2\psi_{4,t})$. 
Hence Eq. (\ref{NLS}) is a system of scalar NLS equation on the intervals $I_j$ coupled
through the Kirchhoff boundary conditions (\ref{continuity})-(\ref{derivatives}) 
included in the domain of $H_{\mathcal{G}}$. \\
On rather general grounds it can be shown that this problem enjoys well-posedness both in strong sense and in the energy space (see in particular \cite[Section 2.6]{CFN17}).\\
%
%
%
%
%
%
%
We want to study standing waves of Eq.~\eqref{NLS}, i.e., its
solutions of the form $\psi_t=e^{-i\omega t}u(x,\omega)$ where $\omega \in \R$
and $u$ is a purely spatial function on $\mathcal{G}$.
In the sequel for the sake of brevity we will often omit the explicit dependence on $\omega$.
Writing the equation component-wise, we get the following scalar
problem: 
\begin{equation}\label{nleigen-pb}
\left\{ 
\begin{array}{ll}
-u^{\prime\prime}_j-u^3_j=\omega u_j, & u_j\in H^2(I_j) \medskip \\ 
u_1(0) = u_2(L) = u_3(0),\quad & u_1(L_1) = u_2(L_1) = u_4(0) \medskip \\ 
u^{\prime}_1(0) - u^{\prime}_2(L) + u^{\prime}_3(0)=0,\quad & 
u^{\prime}_1(L_1) - u^{\prime}_2(L_1) - u^{\prime}_4(0)=0.
\end{array}
\right. 
\end{equation}
In \cite{AST1,AST2} it is shown, among many other things, that the focusing NLS on a double bridge graph has no ground state, i.e. no standing wave exists that minimizes the energy at fixed $L^2$-norm (see also \cite{AST3} for the critical power NLS). In the recent \cite{AST4} information on positive bound states which are not ground states is given. In this paper we are interested, instead, in studying {\it non positive} standing waves profiles. 
\\
We discuss first the case $\omega>0$, taking also the opportunity to fix notations and to recall some elementary but useful facts. It is well known that non vanishing $L^2$ solutions of the stationary focusing NLS on the half-line do not exist. So any solution of our problem is supported on the circle. This further condition forces Dirichlet boundary conditions at the two vertices and makes the above problem (\ref{nleigen-pb}) overdetermined: a solution $u$ belongs to $H_{\mathrm{per}}^{2}([0,L])$ necessarily (see definition (\ref{Hper}) below), and moreover it has to satisfy Dirichlet boundary conditions at $0,L_1$ and $L$. Periodic solutions of stationary NLS on the interval are Jacobi snoidal, cnoidal and dnoidal functions (for a treatise on the Jacobian elliptic functions, we refer e.g. to \cite{handbook, Lawden}). Only cnoidal and dnoidal functions satisfy the {\it focusing} NLS on the circle, and the dnoidal functions do not vanish anywhere and so we rule out them.\\ More precisely, the cnoidal function $v(y):=\cn(y;k)$ with parameter $k$ solves the equation
\begin{equation}
\label{NLScnoidal}
-v''(y)-2k^2v^3(y)=(1-2k^2)v(y)\ .
\end{equation} 
Up to translations it is the only periodic solution of \eqref{NLScnoidal} oscillating around zero and its minimal period is given by 
\begin{equation}
T(k)=4K(k)=4\int_{0}^{1}\frac{dt}{\sqrt{\strut(1-t^2)(1-k^2 t^2)}}
\end{equation}
where $K(k)$ is the so called complete elliptic integral of first kind. 
There results $\cn(0;k)=1$, and $\cn(T/4;k)=0$ gives the first zero in $[0,T]$.
The scaling
\begin{equation}
\label{scaling}
u(x)=\sqrt{2}kpv(px),\ \ \ \ \ \ p=\sqrt{\frac{\omega}{1-2k^2}},\ \ \ \ k\in (0,1/2),\ \ \omega>0
\end{equation}
shows that 
\begin{equation}
u(x)= 
\sqrt{\frac{2\omega k^{2}}{1-2k^{2}}}\ 
\cn \left( \sqrt{\frac{\omega}{1-2k^{2}}}\,x\, ;\, k\right) 
\end{equation} 
solves for every $k\in (0,1/\sqrt{2})$ the equation
\begin{equation}
\label{NLScnoidalomega}
-u''-u^3=\omega u .
\end{equation}
It is a periodic solution of \eqref{NLScnoidalomega} oscillating around zero and its minimal period is given by
\begin{equation}
\label{period}
T_{\omega}(k)=4\sqrt{\frac{1-2k^{2}}{\omega}}K(k).
\end{equation}
Assuming periodicity on $[0,L]$ ($u\in H_{\mathrm{per}}^{2}([0,L])$) gives a countable family of periodic cnoidal functions $u_n$ with parameter $k_n$ defined by the condition that the length of the interval is an integer multiple of the period, $nT_{\omega}(k_n)=L$. A translation of a quarter of period along the circle allows to satisfy the Dirichlet conditions at $0$ and $L$. Up to now we have a sequence of parameters $k_n$ and functions $u_{n,\omega}^\pm$:
\begin{equation}
\label{cnoidalomega}
u_{n,\omega}^\pm(x)=\sqrt{\frac{2\omega k_n^{2}}{1-2k_n^{2}}}\ 
\cn \left( \sqrt{\frac{\omega}{1-2k_n^{2}}}(x \pm T_{\omega}(k_n)/4) ; k_n\right),\quad nT_{\omega}(k_n)=L .
\end{equation} 
Now it is clear that the further Dirichlet condition at $L_1$ can be satisfied if and only if there exists $m<n$ such that $mT_{\omega}(k_n)/2=L_1$, i.e $L_1/L=m/(2n)$ is a rational number. When we add this further condition, an infinite strict subset of the above families of cnoidal functions $u_{n,\omega}^\pm$ satisfies the complete problem \eqref{nleigen-pb} for every positive $\omega$, namely the ones with $n\in\N q_0$, where we denote by $p_0, q_0$ the unique coprime naturals such that $L_1/L=p_0/(2q_0)$ (cf. Remark \ref{q_0}). 
These solutions are supported on the circle and disappear when the length  $L_1$ is not a rational multiple of the length $L$. Moreover, as expected and shown in the Appendix, they bifurcate from the linear eigenvectors of the double bridge quantum graph in the limit of small amplitude.

A similar argument shows that for $\omega=0$ the solutions of \eqref{nleigen-pb} only exist if $L_1/L\in\Q$ 
and form a sequence of suitably rescaled and translated cnoidal functions (cf. \cite{CFN15} for details).

The situation is completely different when we consider solutions with $\omega<0$.
In the first place the above families, which we indicate again as $u_{n,\omega}^\pm$, can be continued to every $\omega<0$ just posing
\begin{equation}
\label{cnoidalomeganeg}
u_{n,\omega}^\pm(x)=\sqrt{\frac{2|\omega| k_n^{2}}{2k_n^{2}-1}}\ 
\cn \left( \sqrt{\frac{|\omega|}{2k_n^{2}-1}}(x \pm T_{\omega}(k_n)/4) ; k_n\right),\ \ \ \  k_n\in (1/\sqrt{2},1).
\end{equation} 
So there is an infinite number of global bifurcation branches $\{(\omega,u_{h q_0,\omega}^\pm:\omega<0\}$, $h\in\N$, originating in correspondence of the linear eigenvalues $\lambda_h$, extending through the range $(-\infty,\lambda_h)$ and compactly supported on the graph. 
We stress again that this infinite family of global bifurcation branches exists only when the ratio $L_1/L$ is rational.
These solutions are the only ones with $u_3=u_4=0$.
However, many more solutions are expected to arise, since for $\omega<0$ non vanishing solutions on the two half-lines are admissible. For example, one can shift any cnoidal solution on the ring, with the results of breaking the continuity at the vertices, and then attach to this shifted cnoidal solution a half-soliton on each tail with the correct height (positive or negative), so to restore continuity. Due to the fact that the half soliton has vanishing derivative at vertex, the Kirchhoff condition is also satisfied. So, at $\omega=0$, from any branch of solutions originating from the linear eigenvalues, a secondary bifurcation branch arises, with non trivial component on the tails (see Fig.\ref{bif}).
This phenomenon, in the simpler example of a tadpole graph (a circle with a single half-line attached), was noticed and studied in \cite{CFN15,[NPS15]}, where several bifurcations and in particular birth of edge solitons and their stability is studied. 
\\ 
Again, such a mechanism of attaching two half-solitons to a shifted cnoidal solution works for every $\omega<0$ if the ratio $L_1/L$ is rational, making the problem nontrivial for irrational ratios.\\
In our main results, we look for solutions of system \eqref{nleigen-pb} and show that they actually exist for every real value of the ratio $L_1/L$. As a matter of fact, a rather complex classification of the general solutions to system \eqref{nleigen-pb} arises for $\omega<0$, requiring in general cnoidal solutions with different parameters $k_1$ and $k_2$ on the two different pieces of the ring, but, precisely in view of this complexity, the study of the complete geography of standing waves branches for $\omega<0$ is postponed to a different paper.\\ The more restricted subject of this paper is the complete description of standing waves  of NLS on  the double bridge graph which exhibit the
following special features:
\begin{itemize}
	\item[(P$_1$)] $u_3$, $u_4$ are nontrivial,
	\item[(P$_2$)] $u_1$, $u_2$ are the restriction to $I_1$, $I_2$ of some $u\in H_{\mathrm{per}}^{2}([0,L])$
\end{itemize}
where 
\begin{equation}
H_{\mathrm{per}}^2([0,L])=\left\{ u\in H^2([0,L]):u(0)=u(L),\ u^{\prime}(0)=u^{\prime }(L)\right\}  
\label{Hper}
\end{equation}
is the second Sobolev space of periodic functions. 
As already remarked, condition (P$_1$) implies $\omega <0$ and 
\begin{equation}
u_j(x)=\pm \sqrt{2|\omega|} \sech\left(\hspace{-0.1cm}\sqrt{|\omega|}(x+a_j)\right),\ \ a_j\in \R,\ \ j=3,4.
\label{sech}
\end{equation}
Condition (P$_2$) implies $u_1^{\prime}(0)-u_2^{\prime}(L)=u_1^{\prime}(L_1)-u_2^{\prime}(L_1)=0$ 
and thus yields $a_j=0$ in (\ref{sech}), by the Kirchhoff conditions.  
Hence we are led to study the solutions $(\omega ,u)$ of the following problem:
\begin{equation}\label{eigenparameter}
\left\{ 
\begin{array}{ll}
-u^{\prime \prime }-u^{3}=\omega u, & u\in H_{per}^{2}([0,L]),\ \ \omega<0 \medskip  \\ 
u(0)=\pm u(L_{1})=\sqrt{2|\omega|} & 
\end{array}
\right.   \tag*{$\left( P_{\pm }\right) $}
\end{equation}
where the sign $\pm $ distinguishes the cases of $u_3$ and $u_4$ with
the same sign (which we may assume positive, thanks to the odd parity of the
equation) or with different signs.
We remark that $(P_\pm)$ is a nonlinear boundary value problem in which the spectral parameter $\omega$ appears explicitly in the boundary conditions. This makes the problem interesting in itself, as spectral parameter dependent (also said "energy dependent") boundary value problems occur frequently in applications. Indeed, they often arise in the passage from a complete system to a reduced system in which the remaining part is eliminated and its effect embodied in a nonstandard boundary condition (see for example \cite{BHLN01,MakThomp12} and reference therein). This is also our case, where the soliton-like nature of the solution on the half-line forces an $\omega$-dependent value of the solution at vertices. 
%
%

Properties of standing waves of the focusing NLS on the double bridge graph satisfying  both (P$_1$) and (P$_2$), or equivalently the solutions to problem $(P_{\pm})$, are described in the following four main theorems. We anticipate that (see proof of Lemma \ref{LEM:pb1}) they cannot be of dnoidal type.
To state the main results, we preliminarily define a function $S\colon (1/\sqrt{2},1)
\rightarrow \R$ by setting 
\begin{equation}
S(k)=4\sqrt{2k^2-1}K(k),  
\label{S:=}
\end{equation}
Notice that $S$ is strictly increasing, continuous and such
that $S\left((1/\sqrt{2},1)\right) =\left(0,+\infty \right)$.
Moreover from now on we denote by $[\ \cdot\ ]$ the floor function, or integer part ($[x]$ is the greatest integer smaller than or equal to the argument $x$). 


The first two results give the classification of standing waves. 
Surprisingly enough, they constitute a countable set if $L_1/L\not\in\Q$ (Theorem \ref{thm1}). If $L_1/L\in\Q$, this set of solutions essentially persists, besides the aforementioned solutions made up of two half-solitons attached to a shifted cnoidal solution. 

\begin{theorem} \label{thm1} 
	Suppose that $L_1/L \in \mathbb{R} \setminus \mathbb{Q}$. The solutions $(\omega ,u)$ of problem 	$\left(P_{\pm}\right)$ 
	are a countable family.
	More precisely, there exist two sequences $\{\omega_{n}^{+}\}_{n\geq 1}$ and $\{\omega_{n}^{-}\}_{n\geq 1}$
	such that the solutions of $\left(P_{\pm}\right)$ are 
	$\left\{ (\omega_{n}^{\pm},u_{n}^{\pm}): n\in\N\right\}$ with 
	\begin{eqnarray}
		u_{n}^{\pm }(x) &:=& \sqrt{\frac{2|\omega _{n}^{\pm }|k_{n}^{2}}{2k_{n}^{2}-1}}\ 
		\cn \left( \sqrt{\frac{|\omega_{n}^{\pm}|}{2k_{n}^{2}-1}}\left(x-s_{n}^{\pm }\right) ; k_{n}\right) ,\quad 
		k_{n}:=S^{-1}\left( L\frac{\sqrt{|\omega _{n}^{\pm }|}}{n}\right) ,  \label{irrational} \\
		s_{n}^{+} &:=&\frac{L}{2n}r_{n},\quad 
		s_{n}^{-}:=\frac{L}{2n}\left( \left|r_{n}\right| -\frac{1}{2}\right) \sgn\left( r_{n}\right) ,\quad 
		r_{n}:=\frac{L_{1}}{L}n-\left[ \frac{L_{1}}{L}n+\frac{1}{2}\right] .  
		\label{kn-sn}
	\end{eqnarray}
\end{theorem}
Note that the last equality of (\ref{irrational}) means that $u_{n}^{\pm}$ has period $L/n$.
The explicit construction of the sequences $\{\omega_{n}^{+}\}$ and $\{\omega_{n}^{-}\}$ is the subject of Lemmas \ref{LEM:equation} and \ref{LEM:equation.2} in Section \ref{SEC:pf1}. They are described as the solution of the equations
\begin{equation*}
\sqrt{| \omega^+_n | }=\frac{n}{L} G\left(\xi_n\right) 
\quad \text{and}\quad 
\sqrt{| \omega^-_n | }=\frac{n}{L} G\left(1-\xi_n\right),\quad n\in\mathbb{N}
\end{equation*}
where $\xi_n:=2|r_n|$ 
and $G$ is a certain monotone function (namely $G=S\circ\varphi^{-1}$, see \eqref{phi} for definition of $\varphi$).

\begin{theorem} \label{thmQ1} 
	Assume that $L_1/L = p/q\in \mathbb{Q}$ with $p,q$ coprime. The set of the solutions to $\left(P_{+}\right)$ is 
	$\{ (\omega,\tilde{u}_{n,\omega}^\pm):\omega<0,\,n\in\N q\} 
	\cup \{ (\omega_{n}^{+},u_{n}^{+}):n\in\N,\,n\notin\N q,\,np/q+1/2\notin\N\}$, 
	where $\omega_{n}^{+},u_{n}^{+}$ are the same of Theorem \ref{thm1}
	and 
	\begin{equation*}
		\tilde{u}_{n,\omega}^\pm(x) :=\sqrt{\frac{2|\omega|k_{n,\omega}^{2}}{2k_{n,\omega}^{2}-1}}\ 
		\cn \left( \sqrt{\frac{|\omega|}{2k_{n,\omega}^{2}-1}}\left(x \pm \gamma_{n,\omega}\right) ; k_{n,\omega}\right) ,
		\quad k_{n,\omega}:=S^{-1}\left( L\frac{\sqrt{|\omega|}}{n}\right) ,  
		\end{equation*}
		\begin{equation}
		\gamma_{n,\omega} = \sqrt{\frac{2k_{n,\omega }^2-1}{\left| \omega \right| }}
		\int_{\sqrt{\frac{2k_{n,\omega}^2-1}{k_{n,\omega}^2}}}^1
	\frac{dt}{\sqrt{\left(1-t^2\right) (1-k_{n,\omega }^2(1-t^2)) }}\,.
	\label{gamma0}
	\end{equation}
	
	The set of the solutions to $\left(P_{-}\right)$ is 
	$\{ (\omega_{n}^{-},u_{n}^{-}):n\in\N,\,n\notin\N q\}$
	if $q$ is odd, 
	and
	$\{ (\omega,\tilde{u}_{n,\omega}^\pm):\omega<0,\,n\in(2\N-1) q/2 \}
	\cup \{ (\omega_{n}^{-},u_{n}^{-}):n\in\N,\,n\notin\N q/2\}$
	if $q$ is even,
	where $\omega_{n}^{-},u_{n}^{-}$ are the same of Theorem \ref{thm1}.
\end{theorem}

Some comments about Theorem \ref{thmQ1} are in order.
First, the functions $\tilde{u}_{n,\omega}^\pm$ have period $L/n$ and constitute the already mentioned secondary bifurcation branches with non trivial components on the tails, arising at $\omega=0$ from any branch of solutions originating from the linear eigenvalues (see Fig.\ref{bif}). 
Note that $\gamma_{n,\omega}$ tends to a quarter of the period, i.e. $L/(4n)$, as $\omega \rightarrow 0$ 
(cf. definition (\ref{gamma}) of $\gamma_{n,\omega}$). 

Second, the solutions $(\omega_{n}^{+},u_{n}^{+})$ in Theorem \ref{thmQ1} are a countable family, since $L_1/L = p/q\in \mathbb{Q}$ implies $\xi_{n+q}=\xi_{n}$ for all $n\in\N$ and therefore $\{|\omega_{n}^{+}|^{1/2}\}$ is a diverging sequence made up of $q$ subsequences 
$\{(mq+1) G(\xi_1)/L\}_{m\geq 0},\ldots,\{(mq+q) G(\xi_q)/L\}_{m\geq 0}$.
We also point out that the solutions $(\omega_{n}^{+},u_{n}^{+})$ 
do not belong to any branch $\{ (\omega,\tilde{u}_{n,\omega}^\pm)\}$, 
since $(\omega_{n}^{+},u_{n}^{+})$ satisfies \eqref{condition1} and Remark \ref{exclude} holds.
Similarly for $(\omega_{n}^{-},u_{n}^{-})$. 

Finally, the content of Theorem \ref{thmQ1} can also, and maybe better, be 
explained in terms of bifurcation diagrams.  
As before, it is convenient to denote by $p_0, q_0$ the unique coprime naturals such that $L_1/L=p_0/(2q_0)$, in such a way that $q_0=q$ if $q$ is odd and $q_0=q/2$ if $q$ is even (cf. Remark \ref{q_0});
note also that $n\notin\N q,\,np/q+1/2\notin\N$ just means $n\notin\N q_0$.
From every bifurcation branch $\{(\omega,{u}_{hq_0,\omega}^\pm):\omega<0\}$ originating from the eigenvalues 
$\lambda_h$, $h\in\N$, a secondary bifurcation branch 
$\{(\omega,\tilde{u}_{hq_0,\omega}^\pm):\omega<0\}$ bifurcates at $\omega=0$ . 
Such a branch of solutions solves $\left(P_{+}\right)$ for all $h$ if $q$ is odd, and $\left(P_{+}\right)$ or $\left(P_{-}\right)$ according as $h$ is even or odd if $q$ is odd. 
Away from these secondary branches, we find isolated solutions to problems $\left(P_{\pm}\right)$ coming from the countable families $\{(\omega_{n}^{\pm},u_{n}^{\pm}):n\in\N\}$, 
namely the ones with $n\notin\N q_0$. 
Observe that we have solutions that oscillate any number of times on $[0,L]$.
This situation is portrayed in Fig.\ref{bif}. 
\begin{figure}
\centering
        \includegraphics [trim=70 130 130 100,scale=0.5]        
        {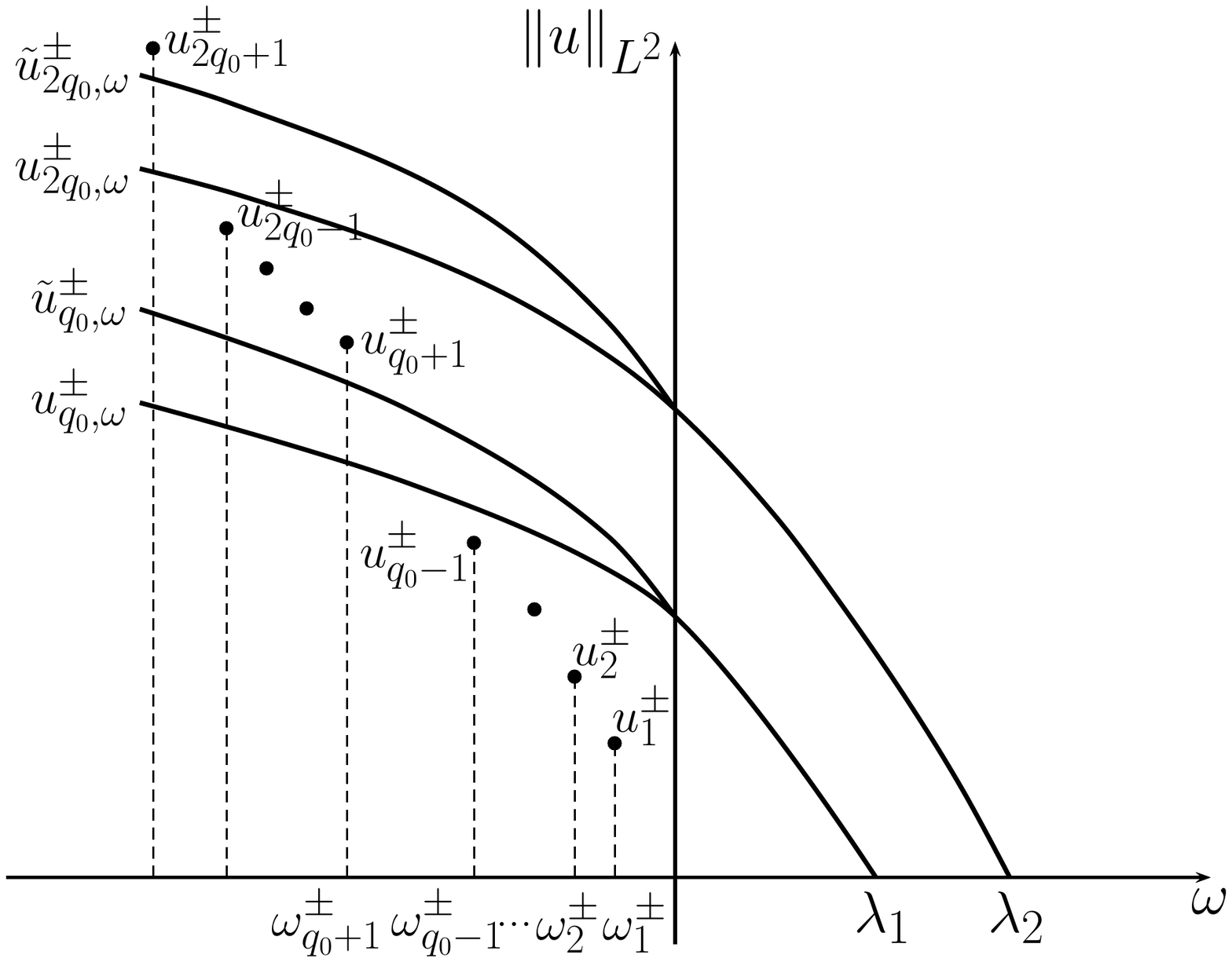}
				\medskip
        \caption{Bifurcation diagram for $L_1/L=p/q=p_0/(2q_0)$ with $p,q$ and $p_0,q_0$ coprime. 
				The functions $\tilde{u}_{hq_0,\omega}^\pm$ solve $\left(P_{+}\right)$ for all $h$ if $q$ is odd, and solve $\left(P_{+}\right)$ or $\left(P_{-}\right)$ according as $h$ is even or odd if $q$ is odd. 
				 All of them have period $L/(hq_o)$, the same of $u_{hq_0,\omega}^\pm$ and the eigenfunctions related to $\lambda_h$.
				The functions $u_n^\pm$, $n\notin\N q_0$, solve $(P_\pm)$ respectively, and have period $L/n$.
				The frequencies $\omega_1^\pm,...,\omega_q^\pm$ need not to be ordered as in the figure.}
        \label{bif}
   \end{figure}

For the sake of completeness, we also represent in Fig.\ref{bif_irr} the sets of the solutions to $(P_\pm)$ for $L_1/L\notin\Q$, as they appear according to Theorems \ref{thm1} and \ref{thm2}.


\begin{remark}
The translation parameters $\gamma_{n,\omega}$ and $s_n^\pm$ of Theorems \ref{thm1} and \ref{thmQ1} are essentially the same, and coincide with the ones needed to match the continuity condition with the half solitons at the vertices. 
More precisely, with the notations of Theorem \ref{thm1} one has $\gamma_{n,\omega_n^\pm}=|s_n^\pm|$ for all $n$ (see Lemmas \ref{LEM:approx} and \ref{LEM:approx.2}).
\end{remark}

The next theorem gives some information about the sequence of frequencies of standing waves pertaining to any irrational geometry: the set of frequencies is unbounded from below and moreover it has at least a finite limit point which, whatever the irrational $L_1/L$ is, is located in a precise interval (see also Remarks \ref{diofantina} and \ref{markov} below).

\begin{theorem}
	\label{thm2} Assume that $L_1/L \in \R \setminus \Q $.
	Then the sequences $\{\omega_{n}^{\pm}\}_{n\geq 1}$ of Theorem \ref{thm1}
	are unbounded from below and have at least a finite cluster point, respectively falling
	in the intervals $I^{\pm}$ defined by 
	\[
	I^{+}=\left[ -\frac{1}{L^{2}}\,K\left( \frac{1}{\sqrt{2}}\right)^{4},0\right] 
	\quad \text{and}\quad 
	I^{-}=\left[ -\frac{16}{5L^{2}}\,K\left(\frac{1}{\sqrt{2}}\right) ^{4},0\right] .
	\]
\end{theorem}

\begin{figure}
\centering
        \includegraphics [trim=70 130 130 100,scale=0.5]        
        {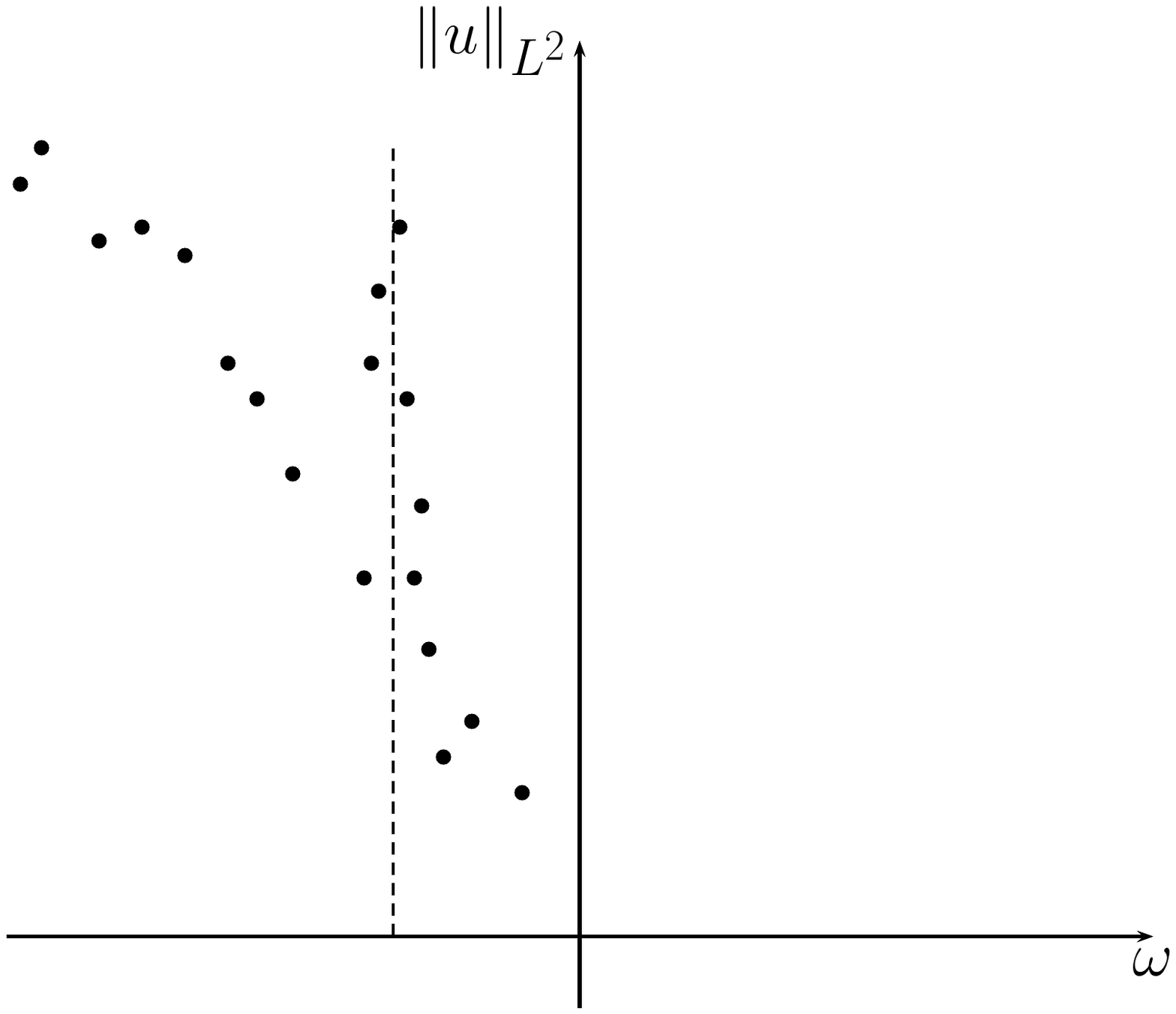}
				\smallskip
        \caption{The appearance of each of the sets of solutions 
				$\{(\omega_n^+,u_n^+)\}_{n\in\N}$ and $\{(\omega_n^-,u_n^-)\}_{n\in\N}$ for $L_1/L\in\R\setminus\Q$, 
				according to Theorems \ref{thm1} and \ref{thm2}.} 
        \label{bif_irr}
   \end{figure}

Finally, the last theorem answers in the affirmative the following inverse problem: can any fixed negative real number be a limit point of standing wave frequencies provided we can choose in a suitable way the ratio $L_1/L$? 

\begin{theorem}\label{thm3}
For every $\omega_0 \leq 0$ there exists a number $L_1/L:=\alpha \in	(0,1)$ such that the frequency sequence 
$\{\omega_{n}^{+}\}_{n\geq 1}$ of Theorem \ref{thm1} has a subsequence converging to $\omega_0$.
\end{theorem}
\noindent
The detailed proofs of the previous theorems fill Sections \ref{SEC:pf1}, \ref{SEC:pf2} and \ref{SEC:pf3}. We only notice that a relevant part of the proofs is played by properties of diophantine approximation of real numbers. Some of them are elementary or well known (as in the case of Dirichlet theorem and Weyl equidistribution theorem) and some other are less. In particular a detailed analysis of the possible cluster points of $\omega_n^{\pm}$ would require information about the so called inhomogeneous diophantine approximation constants of real numbers (see Remark \ref{diofantina}), which are strictly related to the properties of sequence $\xi_n$ (see for example the classical treatise \cite{Cassels}, \cite{Niven} and the more recent \cite{Schmidt}), about which only few precise results are known. However we stress the fact that the analysis here presented is essentially elementary and self contained.\\ Once again, we stress that we are not classifying the totality of standing waves of the double bridge graph, but the subfamily with a periodical component on the ring, which in turn is in $1-1$ correspondence with solutions of boundary value problems $\left(P_{\pm}\right)$. From this point of view, the rather surprising structure of the obtained frequencies, constitutes the {\it nonlinear spectrum} of problems $\left(P_{\pm}\right)$. This seems a new result with an independent interest.\\
In the effort of giving further information on the frequency sequences of Theorem \ref{thm1} and to guess possible directions for rigorous analysis, in the last Section \ref{SEC:num}, some numerical results about the sequence $\{\omega_{n}^{+}\}$ 
are given, with the aid of a simple code run by \texttt{Wolfram MATHEMATICA 10.4.1}. 
The numerics highlights several phenomena. 
In the first place the appearance of  a single or also several cluster points for $\{\omega_{n}^{+}\}$ in the interval $I^{+}$, depending on the choice of different ratios $\alpha$. 
Secondarily, for several choices of algebraic ratios $\alpha$ 
the indices corresponding to the subsequences of $\{\omega_{n}^{+}\}$ converging in $I^{+}$ 
are recognized as distinguished and well known sequences, for example related to Fibonacci or Chebyshev sequences. We do not have at present any clue about this second seemingly curious behavior. We however notice that in principle this is a pure number theoretic property of diophantine approximation constants; it is perhaps noteworthy its appearance in the boundary value problem here studied.
\smallskip

For the convenience of the reader, we collect here some notation. 
\begin{itemize}
	\item $\mathbb{N}$ stands for the set of \emph{positive} integers ($0$ excluded).
	\item $[t]$  is the integer part of $t\in \R$, while $\left\{t\right\} =t-\left[ t\right]$ is the fractional part of $t\in \R$.
\end{itemize}
Finally we notice once and for all that that $L_1/L_2\in \mathbb{R}\setminus \mathbb{Q}$ is equivalent to 
$L_1/L\in \mathbb{R}\setminus \mathbb{Q}$, since 
$\frac{L_1}{L}=\frac{L_1/L_2}{1+L_1/L_2}$.

\section{Proof of Theorems \ref{thm1} and \ref{thmQ1} \label{SEC:pf1}}

With the aim of proving Theorems \ref{thm1} and \ref{thmQ1}, we first solve the auxiliary problem
\begin{equation} \label{pb1}
\left\{ 
\begin{array}{l}
-u''-u^3=\omega u\medskip  \\ 
u\in H_{\mathrm{per}}^2([0,L]),\ \ \omega <0.
\end{array}
\right.   
\end{equation}
Clearly, the solutions of problem $\left(P_{\pm}\right)$ are the solutions of (\ref{pb1}) 
satisfying the boundary condition $u(0) = \pm u(L_1) = \sqrt{2|\omega|}$.
\begin{lemma}\label{LEM:pb1}
	The set of the solutions $(\omega ,u)$ 
	to problem \eqref{pb1} assuming the value $\sqrt{2|\omega|}$
	is the family 
	$\left\{ (\omega ,c_{n,\omega }(\cdot \,;a)):
	n\in \N ,\,\omega <0,\,a\in \left[ -\frac{L}{4n},\frac{3L}{4n}\right) \right\}$,
	where 
	\begin{eqnarray}
	\displaystyle c_{n,\omega }(x;a) &=&
	\sqrt{\frac{2|\omega|k_{n,\omega }^2}{2k_{n,\omega}^2-1}}\cn\left(\sqrt{\frac{|\omega|}{2k_{n,\omega }^2-1}}(x+a);
	k_{n,\omega }\right) ,\medskip   
	\label{c} \\
	\displaystyle k_{n,\omega } &=&S^{-1}\left( L\frac{\sqrt{|\omega |}}{n}\right).
	\label{k}
	\end{eqnarray}
\end{lemma}

\begin{proof}
	The periodic solutions of the equation $-u^{\prime \prime }-u^3=\omega u$ 
	with $\omega \in \R $ are well known and can be expressed in terms of the Jacobian elliptic functions 
	(cf. the discussion in Section \ref{SEC:intro}, and see \cite{CFN15} and references therein). 
	In particular, for $\omega <0$, 
	such solutions are the functions 
	\begin{equation}
	c_{\omega}(x;k,a)=\pm \sqrt{\frac{2|\omega|k^2}{2k^2-1}}\ \cn\left(\sqrt{\frac{|\omega |}{2k^2-1}}(x+a);k\right) 
	\label{pb1_sols}
	\end{equation}
	with $k\in \left( 1/\sqrt{2},1\right) $ and $a\in \R $ free parameters, and
	\begin{equation}
	d_{\omega}(x;k,a)=\pm \sqrt{\frac{2|\omega|}{2-k^2}}\ \dn\left(\sqrt{\frac{|\omega |}{2-k^2}}(x+a);k\right) 
	\label{pb1_sols_dn}
	\end{equation}
	with $k\in [0,1) $ and $a\in \R $ free parameters. 
	The negative sign in (\ref{pb1_sols_dn}) is ruled out, because the corresponding maps only take negative values.
	Moreover, the dnoidal function $\dn$ oscillates between $\sqrt{1-k^2}$ and $1$ 
	and therefore the positive funtions $d_{\omega}$ oscillate between 
	$\sqrt{2|\omega|(1-k^2)}/\sqrt{2-k^2}$ 
	and $\sqrt{2|\omega|}/\sqrt{2-k^2}<\sqrt{2|\omega|}$, which implies that they cannot assume the value $\sqrt{2|\omega|}$.
	So the whole family (\ref{pb1_sols_dn}) is ruled out.
	On the other hand, the function $\cn$ oscillates between $-1$ and $1$ and thus the image of all the functions $c_{\omega}$ contains $\sqrt{2|\omega|}<\sqrt{2|\omega |k^2}/\sqrt{2k^2-1}$, $k\in (1/\sqrt{2},1)$ .
	The period $T$ of $c_{\omega }$ depends on $k$ (and $\omega $) and is given by 
	\[
	T=\frac{S(k)}{\sqrt{|\omega |}}.
	\]
	Hence $c_{\omega}$ belongs to $H_{per}^2([0,L])$ if and only if $L$ is an
	integer multiple of $T$, i.e., $k=S^{-1}(L\sqrt{|\omega |}/n)$ for some $n\in \N $. 
	Therefore the solutions to (\ref{pb1}) assuming the value $\sqrt{2|\omega|}$ are the functions 
	$c_{n,\omega }(x;a)=c_{\omega }(x;k_{n,\omega },a)$ with $n\in \N ,\,\omega <0$ and $a\in \R $.
	Since $c_{n,\omega }(\cdot \,;a)=-c_{n,\omega }(\cdot \,;a-\frac{L}{2n})$, 
	the negative sign in (\ref{pb1_sols}) can be removed in order to avoid duplicate solutions. 
	Finally, the parameter $a$ can be limited to the interval 
	$\left[ -\frac{T}{4},\frac{3T}{4}\right) =\left[ -\frac{L}{4n},\frac{3L}{4n}\right) $ by periodicity. 
\end{proof}
\noindent
Notice that, according to the proof, the function (\ref{c}) has period $L/n$ for every $n,\omega,a$. 

For $n\in \mathbb{N}$ and $\omega <0$, define the auxiliary function 
\[
c_{n,\omega }(x):=c_{n,\omega }\left( x\,;0\right) 
\]
(cf. Fig.\ref{bn}). Observe that 
\[
c_{n,\omega }(\mathbb{R})=
\left[ -\sqrt{\frac{2|\omega |k_{n,\omega }^{2}}{2k_{n,\omega }^{2}-1}},
\sqrt{\frac{2|\omega |k_{n,\omega }^{2}}{2k_{n,\omega}^{2}-1}}\right] 
\quad \textrm{and}\quad 
0<\sqrt{2|\omega|}<\sqrt{\frac{2|\omega |k_{n,\omega }^{2}}{2k_{n,\omega }^{2}-1}}\,, 
\]
since $k_{n,\omega }\in (1/\sqrt{2},1)$. 
Hence $\sqrt{2|\omega|}$ has $2n$ preimages in $\left[ -\frac{L}{4n},L-\frac{L}{4n}\right] $ 
under $c_{n,\omega }$, which we denote by 
\[
x_{1}^{(n,\omega )}<x_{2}^{(n,\omega )}<...<x_{2n}^{(n,\omega )}.
\]
Similarly, $-\sqrt{2|\omega|}$ has $2n$ preimages in $\left[ -\frac{L}{4n},L-\frac{L}{4n}\right] $ 
under $c_{n,\omega }$ as well, which we denote by 
\[
y_{1}^{(n,\omega )}<y_{2}^{(n,\omega )}<...<y_{2n}^{(n,\omega )}.
\]
For future reference, we also set 
\begin{equation}
\gamma _{n,\omega }:=-x_{1}^{(n,\omega )}\quad \left( =x_{2}^{(n,\omega)}>0\right) ,
	\label{gamma}
\end{equation}
in such a way that for $j=1,\ldots ,n$ one has
\begin{equation}
x_{2j-1}^{(n,\omega )}=\left( j-1\right) \frac{L}{n}-\gamma _{n,\omega},\qquad 
x_{2j}^{(n,\omega )}=\left( j-1\right) \frac{L}{n}+\gamma_{n,\omega }  \label{x_j}
\end{equation}
and
\begin{equation}
y_{2j-1}^{(n,\omega )}=\left( 2j-1\right) \frac{L}{2n}-\gamma _{n,\omega},\qquad 
y_{2j}^{(n,\omega )}=\left( 2j-1\right) \frac{L}{2n}+\gamma_{n,\omega }.  \label{y_j}
\end{equation}

Note that definition (\ref{gamma}) is equivalent to (\ref{gamma0}), as we will show at a later stage (see (\ref{d=})).

   

   \begin{figure}
        \includegraphics [scale=0.6]{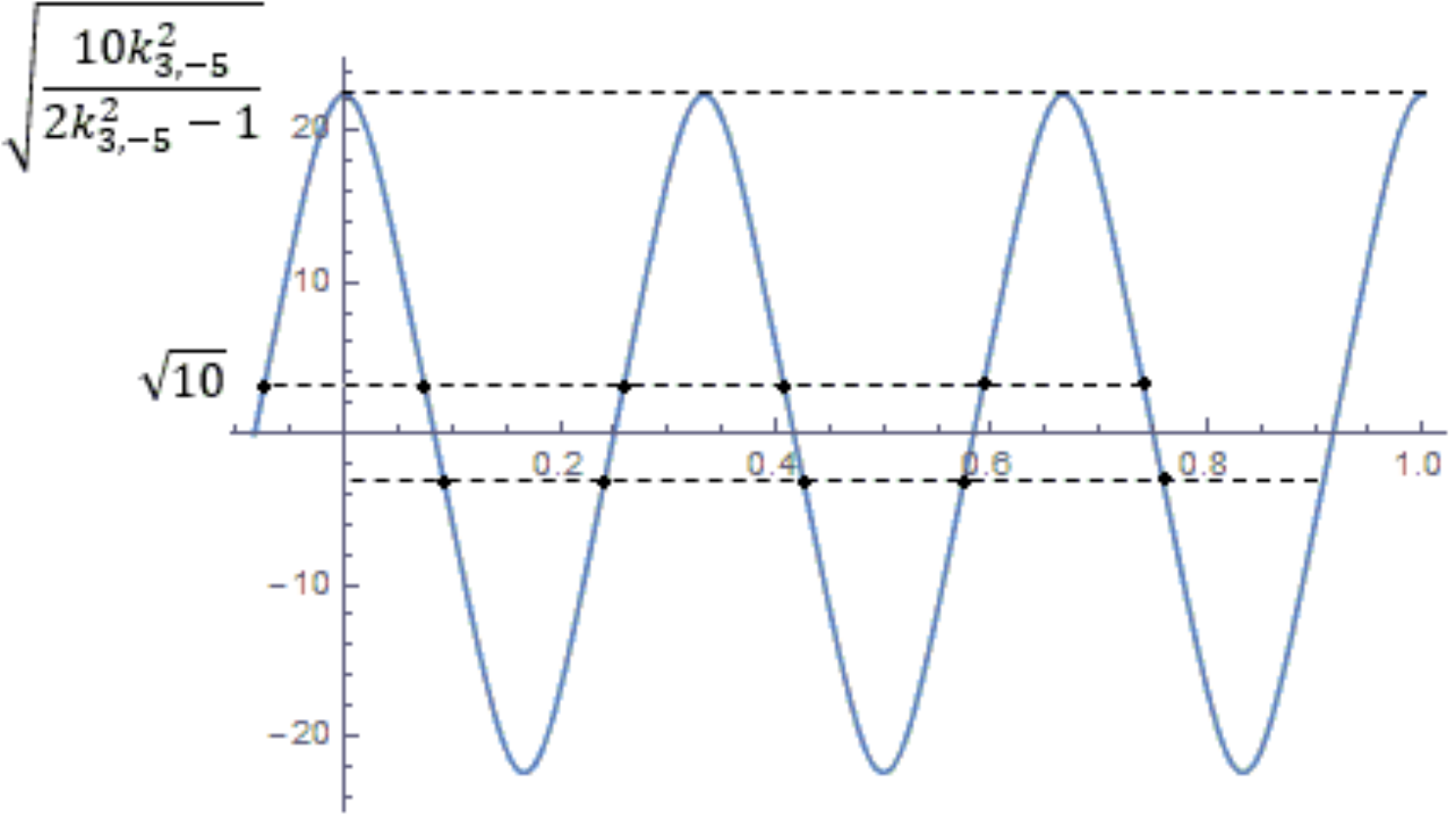}
        \caption{The function $c_{n,\omega}$ on $[-\frac{L}{4n},L]$, with $n=3$, $\omega=-5$ and $L=1$.}
        \label{bn}
\end{figure}


\begin{lemma}
	\label{LEM:approx} 
	A solution $(\omega ,c_{n,\omega }(\cdot \,;a))$ of problem (\ref{pb1}) solves problem $(P_+)$  
	if and only if 
	$a=\pm \gamma_{n,\omega }$ and $n L_1/L \in \N$, or
	\begin{equation}
	a=\pm \gamma_{n,\omega }
		\quad \text{and}\quad 
	\frac{L_1}{L}n-\left[ \frac{L_1}{L}n+\frac{1}{2}\right] =\mp \frac{2n}{L}\gamma_{n,\omega}
	\label{condition1}
	\end{equation}
	(with obvious relation between the signs of the right hand sides), i.e., 
	\[
	a = -s^+_n
	\quad \text{and}\quad 
	\left|\,\frac{L_1}{L}n-\left[ \frac{L_1}{L}n+\frac{1}{2}\right] \right|=\frac{2n}{L}\gamma_{n,\omega}
	\]
	where $s^+_n$ is the defined in (\ref{kn-sn}).
\end{lemma}

\begin{remark} \label{exclude}
\eqref{condition1} and the case with $n L_1/L \in \N$ exclude each other. 
Indeed, $n L_1/L \in \N$ implies 
$\left|n L_1/L-\left[ n L_1/L+1/2\right] \right|=0$ and therefore 
\eqref{condition1} is impossible, since $\gamma_{n,\omega} > 0$.
\end{remark}

\begin{proof}
Denote $m_{n}=\left[ \frac{L_{1}}{L}n+\frac{1}{2}\right] $ for brevity.
If $a=-\gamma_{n,\omega }$, one has
	\[
	c_{n,\omega }(0;a)=c_{n,\omega }\left( 0;-\gamma _{n,\omega }\right)
	=c_{n,\omega }\left( -\gamma _{n,\omega };0\right) 
	=c_{n,\omega}(x_{1}^{(n,\omega )})=\sqrt{2|\omega |}.
	\]
	Moreover there exists $m\in\N$ such that 
	$$c_{n,\omega }(L_{1};a) =c_{n,\omega }(L_1-\gamma_{n,\omega })=c_{n,\omega }(m L/n-\gamma_{n,\omega })
	=c_{n,\omega }(x_1^{(n,\omega )})=\sqrt{2|\omega |}$$ 
	if $n L_1/L \in \N$, 	and one has
	$$c_{n,\omega }(L_{1};a) =c_{n,\omega }(L_{1}-\gamma_{n,\omega })=c_{n,\omega }(2\gamma _{n,\omega }+m_n L/n-\gamma_{n,\omega })
	= c_{n,\omega }(x_2^{(n,\omega )})=\sqrt{2|\omega |}$$
	if (\ref{condition1}) holds.
Hence $(\omega ,c_{n,\omega }(\cdot \,;a))$ solves problem $(P_{+})$. 
The conclusion similarly ensues if $a=\gamma_{n,\omega }$. 
	\\
	Now assume that $(\omega ,c_{n,\omega }(\cdot\,;a))$ solves problem $(P_{+})$. 
	Since $c_{n,\omega }(0;a)=c_{n,\omega }(a)$ and $-\frac{L}{4n}\leq a<\frac{3L}{4n}$, 
	the condition $c_{n,\omega }(0;a)=\sqrt{2|\omega|}$ means 
	\begin{equation}
	\text{either}\quad a=-\gamma _{n,\omega }\quad \text{or}\quad a=\gamma _{n,\omega }\,.
	\label{a}
	\end{equation}
	In the first case, we have 
	$c_{n,\omega }(L_{1};a)=c_{n,\omega}(L_{1}-\gamma _{n,\omega })=c_{n,\omega }(L_{1}+x_{1}^{(n,\omega )})$ with 
	\[
	x_{1}^{(n,\omega )}<L_{1}+x_{1}^{(n,\omega )}<L_{1}\leq\frac{L}{2}<L-\frac{L}{4n}
	\]
	(recall that $0<L_{1}\leq L/2$, since $L_{1}\leq L_{2}$), so that the condition 
	$c_{n,\omega }(L_{1};a)=\sqrt{2|\omega |}$ implies 
	$L_{1}+x_{1}^{(n,\omega)}\in \{x_{2}^{(n,\omega )},x_{3}^{(n,\omega )},\ldots ,x_{2n}^{(n,\omega)}\}$, i.e., 
	\begin{equation}
	L_{1}\in \left\{ x_{2}^{(n,\omega )}-x_{1}^{(n,\omega )},
	x_{3}^{(n,\omega)}-x_{1}^{(n,\omega )},\ldots ,x_{2n}^{(n,\omega )}-x_{1}^{(n,\omega)}\right\} .
	\label{L1}
	\end{equation}
	Recalling (\ref{x_j}), for $j=1,\ldots ,n$ one has
	\begin{eqnarray*}
		x_{2j-1}^{(n,\omega )} =\left( j-1\right) \frac{L}{n}+x_{1}^{(n,\omega )} \ \ \text{and}\ \ \
		x_{2j}^{(n,\omega )} =\left( j-1\right) \frac{L}{n}-x_{1}^{(n,\omega)}
		=\left( j-1\right) \frac{L}{n}+2\gamma _{n,\omega }+x_{1}^{(n,\omega )},
	\end{eqnarray*}
	so that (\ref{L1}) means that there exists $m\in \left\{ 0,...,n-1\right\} $ such that 
	\begin{equation} \label{eq:alt1}
	\text{$L_{1}=m\frac{L}{n}$\quad or \quad $L_{1}=m\frac{L}{n}+2\gamma _{n,\omega}$}.
	\end{equation}
	If the first of such cases occurs, then $m\geq 1$ (since $L_1 \neq 0$) and the proof is complete.
	If the second case holds true, we get 
	$nL_1/L=m+2n\gamma_{n,\omega }/L$ and therefore
	$nL_1/L-1/2<m<nL_1/L+1/2$,
	since $\gamma _{n,\omega }<L/(4n)$.
	This implies $nL_1/L+1/2\notin\N$ and $m=\left[nL_1/L+1/2\right] $,
	which completes the proof again. 
	Finally, if the second alternative of (\ref{a}) holds, a similar argument yields 
	\[
	\text{$L_{1}=m\frac{L}{n}$\quad or \quad $\frac{L_{1}}{L}n=m-\frac{2n}{L}\gamma _{n,\omega }$}
	\]
	instead of (\ref{eq:alt1}), and the conclusion follows as above.
\end{proof}

\begin{lemma}
	\label{LEM:approx.2} A solution $(\omega ,c_{n,\omega }(\cdot \,;a))$ of
	problem (\ref{pb1}) solves problem $\left( P_{-}\right) $ if and only if 
	$a=\pm \gamma_{n,\omega }$ and $n L_1/L +1/2 \in \N$, or
	\begin{equation}
	a=\pm \gamma_{n,\omega }
	 \quad \text{and}\quad
	\frac{L_{1}}{L}n-\left[ \frac{L_{1}}{L}n+\frac{1}{2}\right] =\pm \left( 
	\frac{1}{2}-\frac{2n}{L}\gamma_{n,\omega }\right)
	\label{condition1.2}
	\end{equation}
	(with obvious relation between the signs of the right hand sides), i.e., 
	\[
	a=-s_{n}^{-}
	\quad \text{and}\quad
	\left| \frac{L_{1}}{L}n-\left[ \frac{L_{1}}{L}n+\frac{1}{2}\right] \right| =%
	\frac{1}{2}-\frac{2n}{L}\gamma_{n,\omega }
	\]
	where $s_{n}^{-}$ is the defined in (\ref{kn-sn}).
\end{lemma}

\begin{remark} \label{exclude.2}
Condition \eqref{condition1.2} and the case with $n L_1/L +1/2 \in \N$ exclude each other. 
Indeed, \eqref{condition1.2} with $n L_1/L+1/2 \in \N$ implies 
$\gamma_{n,\omega} = 0$ or $\gamma_{n,\omega} =L/(2n)$, which is impossible since $0<\gamma_{n,\omega} <L/(4n)$.
\end{remark}

\begin{proof}
Denote $m_{n}=\left[ \frac{L_{1}}{L}n+\frac{1}{2}\right] $ for brevity.
If $a=\gamma_{n,\omega }$, one has $c_{n,\omega }(0;a)=\sqrt{2|\omega |}$ as in the proof of Lemma \ref{LEM:approx}.
	Moreover there exists $m\in\N$ such that 
	$$c_{n,\omega }(L_{1};a) =c_{n,\omega }(L_1-\gamma_{n,\omega })
	=c_{n,\omega }\left(\frac{m L}{n}-\frac{L}{2n}-\gamma_{n,\omega }\right)
	=c_{n,\omega }(-y_2^{(n,\omega )})=-\sqrt{2|\omega |}$$ 
	if $n L_1/L \in \N$ (note that $c_{n,\omega }$ is even), and one has
	$$c_{n,\omega }(L_{1};a) =c_{n,\omega }(L_{1}-\gamma_{n,\omega })
	=c_{n,\omega }\left(m_n \frac{L}{n}-\frac{1}{2n}+\gamma_{n,\omega }\right)
	= c_{n,\omega }(-y_1^{(n,\omega )})=-\sqrt{2|\omega |}$$
	if (\ref{condition1}) holds.
This implies that $(\omega ,c_{n,\omega }(\cdot \,;a))$ solves problem $(P_{+})$. 
A similar computation yields the same result if $a=\gamma_{n,\omega }$.

	Now assume that $(\omega ,c_{n,\omega }(\cdot \,;a))$ solves problem $\left( P_{-}\right) $.
	Since 
	$-\frac{L}{4n}\leq a<\frac{3L}{4n}$, 
	condition $c_{n,\omega }(0;a)=\sqrt{2|\omega|}$ means 
	\begin{equation}
	a=-\gamma _{n,\omega }\quad \text{or}\quad a=\gamma _{n,\omega }\,.
	\label{a.2}
	\end{equation}
	In the first case, we have $c_{n,\omega }(L_{1};a)=c_{n,\omega}(L_{1}+x_{1}^{(n,\omega )})$ with 
	\[
	x_{1}^{(n,\omega )}<L_{1}+x_{1}^{(n,\omega )}<L_{1}\leq \frac{L}{2}<L-\frac{L}{4n}
	\]
	(recall that $0<L_{1}\leq L/2$, since $L_{1}\leq L_{2}$), so that condition 
	$c_{n,\omega }(L_{1};a)=-\sqrt{2|\omega|}$ implies 
	$L_{1}+x_{1}^{(n,\omega )}\in\{y_{1}^{(n,\omega )},y_{2}^{(n,\omega )},...,y_{2n}^{(n,\omega )}\}$, i.e., 
	\begin{equation}
	L_{1}\in \left\{ y_{1}^{(n,\omega )}-x_{1}^{(n,\omega )},
	y_{2}^{(n,\omega)}-x_{1}^{(n,\omega )},...,
	y_{2n}^{(n,\omega )}-x_{1}^{(n,\omega )}\right\} .
	\label{L1.2}
	\end{equation}
	Recalling (\ref{y_j}), for $j=1,\ldots ,n$ one has
	\begin{eqnarray*}
		y_{2j-1}^{(n,\omega )} &=&\left( 2j-1\right) \frac{L}{2n}+x_{1}^{(n,\omega)}, \\
		y_{2j}^{(n,\omega )} &=&\left( 2j-1\right) \frac{L}{2n}+\gamma _{n,\omega}
		=\left( 2j-1\right) \frac{L}{2n}+x_{1}^{(n,\omega )}+2\gamma _{n,\omega },
	\end{eqnarray*}
	so that (\ref{L1.2}) means that there exists $j\in \left\{ 1,...,n\right\} $
	such that 
	\[
	L_{1}=\frac{2j-1}{2n}L\quad \text{or}\quad L_{1}=\frac{2j-1}{2n}L+2\gamma_{n,\omega }\,.
	\]
	In the first case, it follows that $nL_1/L+1/2\in\N$ and this completes the proof. 
	In the second case, we get $nL_1/L+1/2=j+2n\gamma_{n,\omega }/L$.
	Since $0<\gamma _{n,\omega }<L/(4n)$, this implies $j<nL_1/L+1/2<j+1/2$
	and therefore $j=\left[ nL_1/L+1/2\right] $, so that we conclude 
	$$
	\frac{L_{1}}{L}n-\left[ \frac{L_{1}}{L}n+\frac{1}{2}\right] 
	=-\left( \frac{1}{2}-\frac{2n}{L}\gamma _{n,\omega }\right) <0.
	$$
	In the second case of (\ref{a.2}), we have 
	$c_{n,\omega}(L_{1};a)=c_{n,\omega }(L_{1}+x_{2}^{(n,\omega )})$ with 
	\[
	x_{2}^{(n,\omega )}<L_{1}+x_{2}^{(n,\omega )}<\frac{L}{2}+\frac{L}{4n}\leq L-\frac{L}{4n},
	\]
	so that the condition 
	$c_{n,\omega }(L_{1};a)=-\sqrt{2|\omega|}$ implies 
	$L_{1}+x_{2}^{(n,\omega )}\in\{y_{1}^{(n,\omega )},y_{2}^{(n,\omega )},...,y_{2n}^{(n,\omega )}\}$, i.e., 
	\begin{equation}
	L_{1}\in \left\{ 
	y_{1}^{(n,\omega )}-x_{2}^{(n,\omega )},
	y_{2}^{(n,\omega)}-x_{2}^{(n,\omega )},...,
	y_{2n}^{(n,\omega )}-x_{2}^{(n,\omega )}\right\} .
	\label{L1bis}
	\end{equation}
	Recalling (\ref{y_j}), for $j=1,\ldots ,n$ one has
	\begin{eqnarray*}
		y_{2j-1}^{(n,\omega )} &=&\left( 2j-1\right) \frac{L}{2n}-\gamma _{n,\omega}
		=\left( 2j-1\right) \frac{L}{2n}-2\gamma _{n,\omega }+x_{2}^{(n,\omega )},
		\\
		y_{2j}^{(n,\omega )} &=&\left( 2j-1\right) \frac{L}{2n}+x_{2}^{(n,\omega )},
	\end{eqnarray*}
	so that (\ref{L1bis}) means that there exists $j\in \left\{ 1,...,n\right\}$
	such that 
	\begin{equation} \label{eq:alt2}
	\text{$L_{1}=\frac{2j-1}{2n}L$\quad or\quad $L_{1}=\frac{2j-1}{2n}L-2\gamma_{n,\omega }$}.
	\end{equation}
	The first of such cases is the same of above, while in the second one 
	there exists $j\in \left\{ 1,...,n\right\} $ such that 
	\begin{equation}
	\frac{L_{1}}{L}n+\frac{1}{2}=j-\frac{2n}{L}\gamma _{n,\omega }.
	\label{L1/L}
	\end{equation}
	Since $0<\gamma _{n,\omega }<\frac{L}{4n}$, this implies $j-1/2<nL_1/L+1/2<j$
	and therefore $j=\left[nL_1/L+1/2\right] +1$, so that (\ref{L1/L}) gives 
	\[
	\frac{L_{1}}{L}n-\left[ \frac{L_{1}}{L}n+\frac{1}{2}\right] 
	=\frac{1}{2}-\frac{2n}{L}\gamma _{n,\omega }>0.
	\]
	This ends the proof.
\end{proof}

For future reference, we define a function $\varphi \colon ( 1/\sqrt{2},1) \rightarrow \R $ by setting 
\begin{equation}\label{phi}
\varphi(k)=\frac{
	\displaystyle\int_{\sqrt{\frac{2k^2-1}{k^2}}}^1\frac{dt}{\sqrt{\strut\left(1-t^2\right)\left(1-k^2(1-t^2)\right)}}}
{\displaystyle\int_0^1\frac{dt}{\sqrt{\strut\left(1-t^2\right) \left(1-k^2 t^2\right) }}}.
\end{equation}
Note that the denominator is the elliptic integral $K(k)$. Such a function $\varphi$ is continuous and strictly decreasing, since 
the denominator $K\left( k\right) $ is positive and strictly increasing and the numerator is positive and strictly decreasing. 
Indeed, one has 
\[
\displaystyle\frac{d}{dk}
\int_{\sqrt{\frac{2k^2-1}{k^2}}}^1\frac{dt}{\sqrt{\strut \left(1-t^2\right)\left( 1-k^2+k^2 t^2\right) }} =
\]
\[
=\int_{\sqrt{\frac{2k^2-1}{k^2}}}^{1}\frac{k \sqrt{1-t^2}}{\left( 1-k^2(1-t^2)\right)^{3/2}}dt
-\frac{1}{k^2\sqrt{\strut \left( 2k^2-1\right) \left( 1-k^2\right) }},
\]
where $k^2\sqrt{2k^2-1}\sqrt{1-k^2}$ has a maximum point on $( 1/\sqrt{2},1)$ for $k^2=(9+\sqrt{17})/16$
and for all $k\in( 1/\sqrt{2},1)$ we have
\begin{eqnarray}
\int_{\sqrt{\frac{2k^2-1}{k^2}}}^{1}\frac{k \sqrt{1-t^2}}{\left( 1-k^2(1-t^2)\right)^{3/2}}dt
&\leq& 
\frac{k \sqrt{1-\frac{2k^2-1}{k^2}}}{\left( 1-k^2(1-\frac{2k^2-1}{k^2})\right)^{3/2}}
\left(1-\sqrt{\frac{2k^2-1}{k^2}}\right) \nonumber \\
&=&\frac{\sqrt{1-k^2}}{k^3} \left(1-\sqrt{2-\frac{1}{k^2}}\right) 
\leq 2
\label{decr}
\end{eqnarray}
(the integrand is decreasing in $t$ and the left-hand side of (\ref{decr}) is decreasing in $k$), so  
\[
\displaystyle\frac{d}{dk}
\int_{\sqrt{\frac{2k^2-1}{k^2}}}^1\frac{dt}{\sqrt{\strut \left(1-t^2\right)\left( 1-k^2+k^2 t^2\right) }} 
\leq 2-\frac{128}{(9+\sqrt{17}) \sqrt{3\sqrt{17}-5} }  < 0.
\]
Moreover one has the two limits
\[
\lim_{k\rightarrow 1^{-}}\varphi \left( k\right) 
=\frac{\displaystyle\int_1^1\frac{dt}{\sqrt{\strut \left(1-t^2\right) t^2}}}{\displaystyle\int_0^1\frac{dt}{1-t^2}}=0\ ,\ \ \ 
\lim_{k\rightarrow \left( \frac{1}{\sqrt{2}}\right)^{+}}\varphi \left(k\right) 
=\frac{\displaystyle\int_0^1\frac{dt}{\sqrt{\strut \left(1-t^2\right) \left( 1+t^2\right) }}}
{\displaystyle\int_0^1\frac{dt}{\sqrt{\strut \left(1-t^2\right) \left( 2-t^2\right) }}}=1,
\]
where the last equality follows from the fact that both numerator and
denominator become 
$
\int_0^{\pi /2}\frac{d\theta}{\sqrt{1+\sin^2 \theta}}
$ 
by the changes of variable $t=\sin \theta $ and $t=\cos\theta $ respectively. 
Hence $\varphi \left( (1/\sqrt{2},1)\right) =\left(0,1\right) $.

\begin{lemma}
	\label{LEM:equation}
	For every $n\in \N $, the equation 
	\begin{equation}
	\left| \,\frac{L_1}{L}n-\left[ \frac{L_1}{L}n+\frac{1}{2}\right]\right| =\frac{2n}{L}\gamma_{n,\omega}
	\label{equation}
	\end{equation}
	has one solution $\omega <0$ if $nL_1/L\notin\N$ and $nL_1/L+1/2\notin\N$, has no solution otherwise.
\end{lemma}

\begin{proof}
	As $\gamma_{n,\omega}=x_2^{(n,\omega )}$ is the unique value in $\left( 0,\frac{L}{2n}\right) $ such that
	$c_{n,\omega }(\gamma_{n,\omega})=\sqrt{2|\omega|}$, i.e.,
	\[
	\cn \left( \sqrt{\frac{\left| \omega \right| }{2k_{n,\omega }^2-1}}\gamma_{n,\omega};k_{n,\omega }\right) =\sqrt{\frac{2k_{n,\omega}^2-1}{k_{n,\omega }^2}}, 
	\]
	one has that
	\begin{eqnarray}
	\gamma_{n,\omega} &=& \sqrt{\frac{2k_{n,\omega }^2-1}{\left| \omega\right| }}
	\arccn\left( \sqrt{\frac{2k_{n,\omega }^2-1}{k_{n,\omega }^2}};k_{n,\omega }\right)  \nonumber \\
	&=& \sqrt{\frac{2k_{n,\omega }^2-1}{\left| \omega \right| }}\int_{\sqrt{\frac{2k_{n,\omega}^2-1}{k_{n,\omega}^2}}}^1
	\frac{dt}{\sqrt{\left(1-t^2\right) \left(1-k_{n,\omega }^2(1-t^2)\right) }}
	\label{d=}
	\end{eqnarray}
	(see \cite{handbook} for the inverse function $\arccn$ of $\cn $ and its representation as an elliptic integral). 
	Since (\ref{k}) and (\ref{S:=}) imply 
	$
	\sqrt{\frac{2k_{n,\omega }^2-1}{\left| \omega \right| }}=\frac{L}{4nK\left( k_{n,\omega }\right) }, 
	$
	from equality (\ref{d=}) we deduce that 
	$\gamma_{n,\omega }=\frac{L}{4n}\varphi \left(k_{n,\omega }\right) $ for all $n\in \N $ and $\omega<0$. 
	Then, recalling the definition (\ref{k}) of $k_{n,\omega }$, equation (\ref{equation}) is equivalent to 
	\begin{equation}
	2\left| \,\frac{L_1 }{L}n-\left[ \frac{L_1 }{L}n+\frac{1}{2}\right]\right| 
	=\varphi \left( S^{-1}\left( \frac{L}{n}\sqrt{\left| \omega \right| }\right) \right) .
	\label{equation 1}
	\end{equation}
	Recalling that $S$ is strictly increasing, continuous and such that 
	$S\left((1/\sqrt{2},1)\right) =\left( 0,+\infty \right) $, 
	the right hand side of (\ref{equation 1}) defines a continuous and strictly increasing function of 
	$\omega $ from $\left( -\infty ,0\right) $ onto $\left( 0,1\right) $. 
	On the other hand, since $t-1<\left[ t\right] \leq t$ for all $t\in \R $, we have 
	\[
	-1\leq 2\left( \frac{L_1 }{L}n-\left[ \frac{L_1 }{L}n+\frac{1}{2}\right]\right) <1, 
	\]
	where the first sign is an equality if and only if $nL_1/L+1/2\in\N$, 
	and the second member vanishes if and only if $nL_1/L\in\N$.
	In these cases, equation (\ref{equation 1}) is impossible.
	Otherwise, we have 
	\[
	2\left| \,\frac{L_1 }{L}n-\left[ \frac{L_1 }{L}n+\frac{1}{2}\right]\right| \in \left( 0,1\right) 
	\]
	and therefore equation (\ref{equation 1}) has
	a unique solution $\omega <0$, which is given by 
	\begin{equation}
	\omega =-\frac{n^2}{L^2}\left[ \left( S\circ \varphi ^{-1}\right) 
	\left(2\left| \,\frac{L_1 }{L}n-\left[ \frac{L_1 }{L}n+\frac{1}{2}\right]\right| \right) \right]^2.
	\label{w}
	\end{equation}
\end{proof}

\begin{lemma}
	\label{LEM:equation.2}For every $n\in \mathbb{N}$, the equation 
	\begin{equation}
	\left| \frac{L_{1}}{L}n-\left[ \frac{L_{1}}{L}n+\frac{1}{2}\right] \right| =%
	\frac{1}{2}-\frac{2n}{L}\gamma_{n,\omega }.  \label{equation.2}
	\end{equation}
	has one solution $\omega <0$ if $nL_1/L\notin\N$ and $nL_1/L+1/2\notin\N$, has no solution otherwise.
\end{lemma}

\begin{proof}
	We argue as in the proof of Lemma \ref{LEM:equation}. 
	Since $\gamma_{n,\omega }=\frac{L}{4n}\varphi \left( k_{n,\omega }\right)$ 
	for all $n\in \mathbb{N}$ and all $\omega <0$, equation (\ref{equation.2}) is equivalent to 
	\[
	2\left| \,\frac{L_{1}}{L}n-\left[ \frac{L_{1}}{L}n+\frac{1}{2}\right]\right| 
	=1-\varphi \left( S^{-1}\left( \frac{L}{n}\sqrt{\left| \omega\right| }\right) \right) 
	\]
	where 
	right hand side defines a continuous and strictly decreasing function of 
	$\omega $ from $\left( -\infty ,0\right) $ onto $\left( 0,1\right) $,
	and left hand side belongs to $\left( 0,1\right) $ if $nL_1/L\notin\N$ and $nL_1/L+1/2\notin\N$, and equals $0$ or $1$ otherwise.
	In this latter case, the equation has no solution. In the former case, it has a unique solution $\omega <0$, given by 
	\begin{equation}
	\omega =-\frac{n^{2}}{L^{2}}\left[ \left( S\circ \varphi ^{-1}\right) 
	\left(1-2\left| \,\frac{L_{1}}{L}n-\left[ \frac{L_{1}}{L}n+\frac{1}{2}\right]\right| \right) 
	\right] ^{2}.  
	\label{w.2}
	\end{equation}
\end{proof}%

\begin{proof}[Proof of Theorems \ref{thm1} and \ref{thmQ1}]
	The conclusions easily follow from Lemmas \ref{LEM:pb1}-\ref{LEM:equation.2}, so we just make some remarks.
	Concerning the case $L_1/L\notin\Q$, we point out that conditions $nL_1/L\notin\N$ and $nL_1/L+1/2\notin\N$ are always true and 
	the fact that the set of solutions is countable (not finite) follows from Theorem \ref{thm2}, where we show that the sequences $\{\omega_n^\pm\}$ are unbounded below.
	As to the case $L_1/L=p/q\in\Q$ with $p,q$ coprime, we observe that condition $n L_1/L\in\N$ is equivalent to $n\in\N q$, 
	whereas condition $n L_1/L+1/2\in\N$ is impossible if $q$ is odd and amounts to $n\in(2\N-1) q/2$ if $q$ is even.
	\end{proof}


%

\section{Proof of Theorem \ref{thm2}      \label{SEC:pf2}}

This section is devoted to the proof of Theorem \ref{thm2}, 
so assume $L_1 /L_2 \in \R \setminus \Q $  
and let $\{\omega^\pm_n\}_{n\geq 1}$ be the sequences of Theorem \ref{thm1}.

According to the proof of Lemmas \ref{LEM:equation} and \ref{LEM:equation.2} 
(see in particular (\ref{w}) and (\ref{w.2})), for all $n\geq 1$ we have 
\[
\sqrt{| \omega^+_n | }=\frac{n}{L} G\left(\xi_n\right) 
\quad \text{and}\quad 
\sqrt{| \omega^-_n | }=\frac{n}{L} G\left(1-\xi_n\right),
\]
where 
\[
G:=S\circ \varphi^{-1}\quad \text{and}\quad 
\xi_n:=2\left| \,\frac{L_1 }{L}n-\left[ \frac{L_1 }{L} n+\frac{1}{2}\right] \right| .
\]
Clearly, with a view to proving Theorem \ref{thm2}, we can equivalently
study the limit points of $\{\hspace{-0.1cm} \sqrt{| \omega^\pm_n|}\}$. 
In doing this, we will exploit some well known results from the metric theory of
Diophantine approximations, for which we refer to \cite{Cassels,Niven,Schmidt}.

Denote 
\[
\alpha =\frac{L_1}{L}
\]
for brevity and observe that 
\[
\xi_n=2\left| \,\alpha n-\left[ \alpha n+\frac{1}{2}\right] \right| 
=\left|\,2\left\{ \alpha n+\frac{1}{2}\right\} -1\right| \in \left( 0,1\right) . 
\]
Here and in the following, $\left\{t\right\} =t-\left[ t\right] $ denotes
the fractional part of $t\in \R $. Note that the cases $\xi_n=0$ and $\xi_n=1$ 
are ruled out because $\alpha \notin \Q $.

\begin{lemma}
	\label{LEM:unbdd}
	The sequence $\{\hspace{-0.1cm} \sqrt{| \omega^+_n|}\}$ is unbounded.
\end{lemma}

\begin{proof}%
	By contradiction, assume that there exists a constant $c>0$ such that 
	$\sqrt{| \omega^+_n | }\leq c$ for all $n\in \N $. 
	Since $S$ is increasing and $\varphi $ is decreasing and positive, $G^{-1}$ is decreasing and positive and thus we get 
	\begin{equation}
	\xi_n=G^{-1}\left( \frac{L}{n}\sqrt{| \omega^+_n | }\right) \geq G^{-1}\left( Lc\right) >0
	\quad \text{for all }n\in \N .
	\label{bddaway}
	\end{equation}
	On the other hand, since $\alpha \notin \Q $, by the Dirichlet's
	approximation theorem (see e.g. \cite[Theorem 1A and Corollary 1B]{Schmidt})
	there exist infinitely many rational numbers $m/n$ such that 
	\begin{equation}
	\left| \,\alpha -\frac{m}{n}\right| <\frac{1}{n^2}.
	\label{approx}
	\end{equation}
	This amounts to $m\in \left( \alpha n-\frac{1}{n},\alpha n+\frac{1}{n}\right) $ 
	where the right hand side interval has length $2/n$ and is centered in the irrational number $\alpha$, so that necessarily $m=\left[ \alpha n+\frac{1}{2}\right] $ if $n\geq 2$. 
	The set of the denominators of the rationals $m/n$ must be infinite 
	(otherwise, (\ref{approx}) implies that the set of the numerators is also finite) 
	and we may arrange them in a divergent sequence $\{n_j\}$ such that $n_j\geq 2$. 
	Hence for all $j$ we get
	\[
	\left| \,\alpha -\frac{1}{n_j}\left[ \alpha n_j+\frac{1}{2}\right]\right| <\frac{1}{n_j^2}
	\]
	and therefore 
	\[
	\xi_{n_j}=2\left| \,\alpha n_j-\left[ \alpha n_j+\frac{1}{2}\right]\right| <\frac{2}{n_j}.
	\]
	This is a contradiction, since $\{\xi_{n_j}\}$ is bounded away from zero by (\ref{bddaway}).%
\end{proof}

In order to investigate the existence of finite cluster points for $\{\hspace{-0.1cm} \sqrt{| \omega^+_n|}\}$, 
we observe that they can only come from subsequences of $\{\xi_n\}$ converging to $1$. 
Indeed, recalling the properties of the functions $S$ and $\varphi $, the function $G$ is
continuous and strictly decreasing from $\left( 0,1\right) $ onto $\left(0,+\infty \right) $ 
and therefore $\xi_{n_j}\rightarrow \ell \in \left[0,1\right) $ implies 
$\sqrt{|\omega^+_{n_j}|}=n_j G(\xi_{n_j})/L\rightarrow +\infty $. 
On the other hand, if $\xi_{n_j}\rightarrow 1$, then $G(\xi_{n_j})\rightarrow 0$ 
and the behaviour of $\sqrt{|\omega^+_{n_j}|}$ depends on the rate of the infinitesimal $G(\xi_{n_j})$. 
Note that such a case actually occurs, since the Weyl criterion for uniformly distributed
sequences (see e.g. \cite[page 66]{Cassels}) assures that the sequence 
$\left\{\left\{ \alpha n+\frac{1}{2}\right\} \right\}_{n \geq 1}$ is dense in $\left[ 0,1\right] $
and therefore it admits subsequences converging both to $0$ and to $1$, to
each of which there correspond a subsequence of $\{\xi_n\}$ converging to $1$.

\begin{lemma}
	\label{LEM:asymptotics}
	Let $\{\xi_{n_j}\}$ be any subsequence of $\{\xi_n\}$ such that $\xi_{n_j}\rightarrow 1$. Then 
	\[
	\sqrt{| \omega^+_{n_j}| }\sim \frac{2}{L}\,K\left( \frac{1}{\sqrt{2}}\right)^2 n_{j}\left(1-\xi_{n_j}\right) 
	\quad \text{as } j\rightarrow \infty .
	\]
\end{lemma}

Here and in the following, $\sim $ denotes the asymptotic equivalence of
functions ($f\sim g\Leftrightarrow f=g+o\left( g\right) $).\medskip

\begin{proof}%
	We want to estimate the rate at which $G(t)=S(\varphi ^{-1}(t))\rightarrow 0$
	as $t\rightarrow 1^{-}$. We have that $\varphi^{-1}(t)\rightarrow (1/\sqrt{2})^{+}$ as $t\rightarrow 1^{-}$, 
	whence 
	\begin{equation}
	G(t)=S(\varphi^{-1}(t))
	\sim 4\sqrt[4]{8}\,K\left( \frac{1}{\sqrt{2}}\right)\left( \varphi^{-1}(t)-\frac{1}{\sqrt{2}}\right) ^{1/2}
	\quad \text{as } t\rightarrow 1^{-}.
	\label{Geq}
	\end{equation}
	Now denote $\varphi (k)=H(k)/K(k)$, with $H(k)$ given by the numerator of
	definition (\ref{phi}). As $k\rightarrow (1/\sqrt{2})^{+}$, both $H(k)$ and $K(k)$ 
	converge to $K\left( 1/\sqrt{2}\right) $ and we have 
	\[
	K^{\prime }\left( k\right) 
	=\int_0^1\frac{kt^{2}dt}{\sqrt{\strut (1-t^2)(1-k^2 t^2)^3}}\rightarrow 
	\int_0^1\frac{2t^2 dt}{\sqrt{\strut (1-t^2)(2-t^2)^3}}\in \R \setminus \left\{ 0\right\} 
	\]
	and 
	\[
	H^{\prime }\left( k\right) 
	=\int_{\sqrt{\frac{2k^2-1}{k^2}}}^{1}\frac{k \sqrt{1-t^2}}{\left( 1-k^2(1-t^2)\right)^{3/2}}dt
	-\frac{1}{k^2\sqrt{\left( 1-k^2\right)\left(\sqrt{2} k+1\right) \sqrt{2} \left(k-1/\sqrt{2}\right) }},
	\]
	where 
	\[
	\int_{\sqrt{\frac{2k^2-1}{k^2}}}^{1}\frac{k\sqrt{1-t^2}dt}{\left(1-k^2(1-t^2)\right)^{3/2}}\rightarrow 
	\int_0^1\frac{2\sqrt{1-t^2}dt}{\left(1+t^2\right)^{3/2}}\in \R \setminus \left\{ 0\right\} 
	\]
	and 
	\[
	\frac{1}{k^2\sqrt{\left( 1-k^2\right)\left(\sqrt{2} k+1\right) \sqrt{2} }}\rightarrow \sqrt[4]{8}.
	\]
	Hence 
	\[
	\varphi^{\prime }(k)=
	\frac{H^{\prime }\left( k\right) K\left( k\right)-K^{\prime }\left( k\right) H\left( k\right) }{K\left( k\right)^2}
	\sim - \frac{\sqrt[4]{8}}{K\left( \frac{1}{\sqrt{2}}\right) }\frac{1}{\left( k-\frac{1}{\sqrt{2}}\right)^{1/2}}
	\quad \text{as }k\rightarrow \left( \frac{1}{\sqrt{2}}\right)^{+} 
	\]
	and therefore 
	\[
	\lim_{k\rightarrow (1/\sqrt{2})^{+}}\frac{1-\varphi (k)}{\left( k-\frac{1}{\sqrt{2}}\right)^{1/2}}
	=\lim_{k\rightarrow (1/\sqrt{2})^{+}}\frac{-\varphi^{\prime }(k)}{\frac{1}{2}\left( k-\frac{1}{\sqrt{2}}\right)^{-1/2}}
	=\frac{2\sqrt[4]{8}}{K\left( \frac{1}{\sqrt{2}}\right) }. 
	\]
	This implies 
	\[
	\lim_{t\rightarrow 1^{-}}\frac{\varphi^{-1}(t)-\frac{1}{\sqrt{2}}}{\left(1-t\right)^2}
	=\lim_{k\rightarrow (1/\sqrt{2})^{+}}\frac{k-\frac{1}{\sqrt{2 }}}{\left( 1-\varphi (k)\right)^2}
	=\frac{K\left( \frac{1}{\sqrt{2}}\right)^2}{8\sqrt{2}}, 
	\]
	i.e., 
	\begin{equation}
	\varphi^{-1}(t)-\frac{1}{\sqrt{2}}\sim \frac{K\left( \frac{1}{\sqrt{2}}\right)^2}{8\sqrt{2}}\left( 1-t\right)^2
	\quad \text{as } t\rightarrow 1^{-}. 
	\label{phi-1eq}
	\end{equation}
	The result then follows from (\ref{Geq}) and (\ref{phi-1eq}).%
\end{proof}%

According to the last lemma, the problem of the finite cluster points of 
$\{\hspace{-0.1cm} \sqrt{| \omega^+_n|}\}$ is reduced to the one of the convergent
subsequences of $\left\{ n\left( 1-\xi_n\right) \right\} $. Notice that 
\begin{equation}
1-\xi_n=2\left| \left\{ \alpha n\right\} -\frac{1}{2}\right| 
\quad \text{for all }n\in \N .
\label{1-an}
\end{equation}

\begin{lemma}
	\label{LEM:cluster1}
	There exist infinitely many indices $n$ such that 
	\[
	0\leq \sqrt{| \omega^+_n | }\leq \frac{1}{L}\,K\left( \frac{1}{\sqrt{2}}\right)^2.
	\]
\end{lemma}

\begin{proof}%
	As $\alpha $ is irrational, there exist infinitely many $n,m\in \mathbb{Z}$
	such that 
	\[
	\left| n\left( \alpha n+\frac{1}{2}-m\right) \right| <\frac{1}{4}
	\quad \text{and}\quad \left| \alpha n+\frac{1}{2}-m\right| <\frac{1}{2} 
	\]
	(see e.g. \cite[Corollary 2.4]{Niven}). The second inequality ensures that 
	\[
	\left| \alpha n+\frac{1}{2}-m\right| =\displaystyle \min_{k\in \mathbb{Z}}\textstyle \left| \alpha n+\frac{1}{2}-k\right|,
	\] 
	i.e., 
	\[
	\left| \alpha n+\frac{1}{2}-m\right| =\left\| \alpha n+\frac{1}{2}\right\| , 
	\]
	where $\left\| t \right\| := \min_{k\in \mathbb{Z}}\textstyle \left| t-k\right|$ denotes the distance from $\mathbb{Z}$.
	Hence the first inequality says that there exist infinitely many $n\in \mathbb{Z}$ 
	such that $n\left\| \alpha n+\frac{1}{2}\right\| <\frac{1}{4}$.
	Since $\left\| \alpha (-n)+\frac{1}{2}\right\| =\left\| \alpha n-\frac{1}{2}\right\| =\left\| \alpha n+\frac{1}{2}\right\| $, 
	we may assume that such integers $n$ are positive, so that we conclude that 
	\begin{equation}
	n\left\| \alpha n+\frac{1}{2}\right\| <\frac{1}{4}
	\quad \text{for infinitely many }n\in \N \text{.}
	\label{niven}
	\end{equation}
	Now observe that, for every $n\in \N $ there exists $m\in \N $ such that 
	\[
	m-\frac{1}{2}<\alpha n<m
	\quad \text{or}\quad 
	m<\alpha n<m+\frac{1}{2}. 
	\]
	In the first case, one has 
	\[
	\left\| \alpha n+\frac{1}{2}\right\| 
	=\alpha n+\frac{1}{2}-m,\quad \left[\alpha n\right]
	=m-1,\quad \left\{ \alpha n\right\} > \frac{1}{2} 
	\]
	and therefore 
	$\left\| \alpha n+\frac{1}{2}\right\| =\alpha n+\frac{1}{2}-\left[ \alpha n\right]-1
	=\left\{ \alpha n\right\} -\frac{1}{2}>0$. 
	In the second case, we get 
	\[
	\left\| \alpha n+\frac{1}{2}\right\| =m+1-\left( \alpha n+\frac{1}{2}\right),\quad 
	\left[ \alpha n\right] =m,\quad 
	\left\{ \alpha n\right\} <\frac{1}{2} 
	\]
	and hence 
	$\left\| \alpha n+\frac{1}{2}\right\| =\left[ \alpha n\right]-\alpha n+\frac{1}{2}
	=\frac{1}{2}-\left\{ \alpha n\right\} >0$. 
	So, in any case, we obtain 
	\begin{equation}
	\left\| \alpha n+\frac{1}{2}\right\| =\left| \left\{ \alpha n\right\} -\frac{1}{2}\right| 
	\quad \text{for all }n\in \N 
	\label{distance=}
	\end{equation}
	and the conclusion follows from (\ref{niven}), (\ref{1-an}) and Lemma \ref{LEM:asymptotics}.%
\end{proof}%


\begin{lemma}
	\label{LEM:unbdd.2}
	The sequence $\{\hspace{-0.1cm} \sqrt{| \omega^-_n|}\}$ is unbounded.
\end{lemma}

\begin{proof}
	Assuming by contradiction that $\{\hspace{-0.1cm} \sqrt{| \omega^-_n|}\}$ is bounded, 
	as in the proof of Lemma \ref{LEM:unbdd} we get that 
	$1-\xi_{n}=G^{-1}( L\sqrt{| \omega^-_n | }/n ) $ 
	is bounded away from zero. But this is false, since, as in the
	proof of Lemma \ref{LEM:cluster1}, there exist infinitely many $n\in \mathbb{N}$ such that 
	\[
	1-\xi_{n}=2\left| \left\{ \alpha n\right\} -\frac{1}{2}\right| \leq \frac{1}{2n}. 
	\]
\end{proof}

The possible finite cluster points of $\{\hspace{-0.1cm} \sqrt{| \omega^-_n|}\}$ can only come
from subsequences of $\{\xi_{n}\}$ converging to $0$, since 
$\xi_{n_{j}}\rightarrow \ell \in \left( 0,1\right] $ implies 
$\sqrt{|\omega^-_{n_{j}}|}=n_{j}G(1-\xi_{n_{j}})/L\rightarrow +\infty $, 
whereas, if $\xi_{n_{j}}\rightarrow 0$, the behaviour of $\sqrt{|\omega^-_{n_{j}}|}$ 
depends on the rate of the infinitesimal $G(1-\xi_{n_{j}})$. 
Note that the Weyl criterion for uniformly distributed sequences (see e.g. \cite[page 66]{Cassels}) 
assures that the sequence $\left\{\left\{ \alpha n+\frac{1}{2}\right\}\right\} $ is dense in $\left[ 0,1\right] $ 
and thus admits subsequences converging to $1/2$, to each of which there corresponds a subsequence 
$\{\xi_{n_{j}}\}$ of $\{\xi_{n}\}$ converging to $0$. As in Lemma \ref{LEM:asymptotics}, for such a subsequence we get that 
\begin{equation}
\sqrt{| \omega^-_{n_{j}}| }\sim \frac{2}{L}\,K\left(\frac{1}{\sqrt{2}}\right) ^{2}n_{j}\xi_{n_{j}}
\quad \text{as }j\rightarrow \infty .  
\label{stima.2}
\end{equation}

\begin{lemma}
	\label{LEM:cluster1.2}
	There exist infinitely many indices $n$ such that 
	\begin{equation} \label{I-}
	0\leq \sqrt{| \omega^-_n |}\leq \frac{4}{L\sqrt{5}}\,K\left(\frac{1}{\sqrt{2}}\right)^{2}. 
	\end{equation}
\end{lemma}

\begin{proof}
	Since $\alpha $ is irrational, by the Hurwitz approximation theorem (see e.g. \cite[Theorem 1.5]{Niven}) 
	there exist infinitely many rational numbers $m/n$ such that 
	\begin{equation}    \label{Hurwitz}
	\left| \,\alpha -\frac{m}{n}\right| <\frac{1}{\sqrt{5}\,n^2}.
	\end{equation}
	Hence, as in the proof of Lemma \ref{LEM:unbdd}, there exists a
	diverging sequence of indexes $(n_{j})$ such that for all $j$ we have 
	\[
	\left| \,\alpha -\frac{1}{n_{j}}\left[ \alpha n_{j}+\frac{1}{2}\right]\right| 
	< \frac{1}{\sqrt{5}\,n_{j}^{2}} 
	\]
	and therefore 
	\[
	n_{j}\xi_{n_{j}}
	= 2n_{j}^{2}\left| \,\alpha -\frac{1}{n_{j}}\left[ \alpha n_{j}+\frac{1}{2}\right] \right| 
	< \frac{2}{\sqrt{5}}. 
	\]
	The conclusion then follows from (\ref{stima.2}).%
\end{proof}

\begin{proof}[Proof of Theorem \ref{thm2}]
	It readily follows from Lemmas \ref{LEM:unbdd}, \ref{LEM:cluster1}, \ref{LEM:unbdd.2} and \ref{LEM:cluster1.2}.
\end{proof}%

\begin{remark}\label{diofantina}
	In Diophantine Analysis, one introduces the quantities
	\[
	{\mathcal M}_\pm(\alpha,\beta)=\liminf_{n\to \infty} n\left\| \pm\alpha n-\beta \right\|,\ \ \ \ \ \ n\in \N,
	\]
	called {\it one side inhomogeneous diophantine approximation constants} 
	({\it homogeneous} for $\beta=0$). They measure how well multiples of a fixed irrational $\alpha$ approximate a real $\beta$. 
	\noindent
	From our point of view
	\begin{itemize}
		\item if ${\mathcal M}_+(\alpha,-\frac{1}{2})=0$, the sequence ${\omega_n^+}$ accumulates to zero;
		\item if ${\mathcal M}_+(\alpha,-\frac{1}{2})\neq 0$, the sequence ${\omega_n^+}$ has a non trivial limit point and it does not accumulate to zero.
	\end{itemize}
The main classical result about the inhomogeneous diophantine approximation constant is the following (Minkowski 1901, Kintchine 1935, Cassels 1954): 
\begin{center}
\emph{
	For every $\beta\notin \mathbb{Z}+\alpha \mathbb{Z}$ one has ${\mathcal M}_-(\alpha,\beta)={\mathcal M}_+(\alpha,\beta)\leq\frac{1}{4}$.
	}
	\end{center}
It is possible to construct a real number~$\alpha=L_1/L$ such that ${\mathcal M}_+(\alpha,-1/2)$ is exactly known, below the threshold $1/4$ and away from $0$; for example, $M_+(\sqrt{e},-1/2)=1/8$ (see \cite{Komatsu97,Pinner01} and references therein);
it is also possible to construct a real number $\alpha=L_1/L$ such that ${\mathcal M}_+(\alpha,-1/2)=0$. 
\end{remark}

\begin{remark} \label{markov}
The right hand side of \eqref{I-} depends on the Hurwitz approximation theorem (in particular on \eqref{Hurwitz}), which holds true for every irrational $\alpha$. 
If we exclude classes of irrationals, estimate \eqref{Hurwitz} can be refined by replacing $\sqrt{5}$ with bigger constants and thus the interval $I^-$ becomes smaller, giving a more accurate localization of the cluster point of Theorem \ref{thm2}. 
For example, excluding the irrational $(\sqrt{5}-1)/2$ and all the numbers equivalent to it in a suitable sense, \eqref{Hurwitz} holds with $\sqrt{8}$ instead of $\sqrt{5}$. 
This theory can be found in full detail in \cite[Chapter 2]{Cassels}.
\end{remark}

\section{Proof of Theorem \ref{thm3} \label{SEC:pf3}}


Thanks to Lemma \ref{LEM:asymptotics}, we can prove Theorem \ref{thm3} by showing that for every $\ell\geq 0$ there exists $\alpha \in(0,1)$ such that the sequence $\{n\,(1-\xi_n)\}$ has a subsequence converging to $\ell$.
As $\ell$ is arbitrary and according to (\ref{1-an}) and the proof of Lemma \ref{LEM:cluster1} (see in particular (\ref{distance=})), this is equivalent to show that 
\begin{equation} \label{xi_bar:=}
\tilde{\xi}_n := n \left( \{n \alpha \}-\frac{1}{2} \right)
\end{equation}
has a subsequence converging to $\ell$.

It is natural to construct the number $\alpha$ by looking at its
expansion in a fixed base, let us say $2$. Hence we need to construct a number
\[
\alpha = 0.b_1b_2b_3\ldots
\]
where each $b_j \in \{0,1\}$. The idea of the proof is to observe
that, in the binary system, a multiplication by a power of $2$ simply
moves the ``binary point'' to the right.

To simplify the construction, let us
assume that our limit $\ell$ is an integer. To represent such an
integer in the binary system we need, say, $n_1-1$ digits.

Now, let us fix arbitrarily $b_1$, $b_2$, \ldots, $b_{n_1}$. When we
construct the number
\[
2^{n_1} \{ 2^{n_1} \alpha \} ,
\]
the ``block'' of digits from the position $n_1+1$ to the position $2n_1$ are moved on the left hand side of the binary point. We choose these digits so that they represent $\ell+2^{n_1-1}$, and we set
\[
b_{2n_1+1}=\ldots=b_{3n_1}=0.
\]
Now we pick $n_2 = 3 n_1$ and we repeat the construction. We begin with a block of $n_2$ arbitrary digits as before, where $b_{n_2+1}$ plays the same r\^{o}le as $b_1$. We insert a block of $n_2$ digits that represent $\ell+2^{n_2-1}$ and a block of $n_2$ null digits.
It is easy to realize that
$\xi_{2^{n_j}} =2^{n_j} \{2^{n_j} \alpha \} - 2^{{n_j}-1}\to \ell$ as $j
\to +\infty$.

If $\ell$ is a real number, the first step of the previous case
``fixes'' its integer part with $n_1$ digits. In the second step, we
insert a longer ``block'' of arbitrary digits, whose length is the
number of digits of the integer part of $\ell$ plus one. After this
arbitrary block, we put a block consisting of the integer part of
$\ell$ \emph{glued to} the first digit of its fractional part, and
another block of zeroes of the same length.  We repeat again to
``fix'' the second digit of the fractional part of $\ell$, and so on.

We remark that now the \emph{length} of each new block of non-trivial
numbers increases, since the binary expansion of $\ell$ may contain
infinitely many digits to ``fix''. This procedure produces a sequence
$n_1<n_2<\cdots$ of integers such that $\xi_{n_j} =2^{n_j} \{2^{n_j}
\alpha \} - 2^{{n_j}-1}\to \ell$ as $j \to +\infty$.

More precisely, let us write
\[
\alpha = \frac{b_1}{2} + \frac{b_2}{2^2} + \ldots
\]
where each $b_j \in \{0,1\}$. Hence
\[
2^n \alpha = b_1 2^{n-1} + b_2 2^{n-2} + \cdots + b_n 2^0 + \frac{b_{n+1}}{2} + \frac{b_{n+2}}{2^2} + \cdots
\]
so that
\[
\left\{ 2^n \alpha \right\} = \frac{b_{n+1}}{2} + \frac{b_{n+2}}{2^2} + \cdots
\]
and
\[
2^n \left\{ 2^n \alpha \right\} = b_{n+1} 2^{n-1} + b_{n+2} 2^{n-2} + \cdots +
b_{2n} + \frac{b_{2n+1}}{2} + \frac{b_{2n+2}}{2^2} + \cdots
\]
Therefore
\[
2^n \left\{ 2^n \alpha \right\} -2^{n-1}=\left(b_{n+1}-1 \right) 2^{n-1} + b_{n+2} 2^{n-2} + \cdots +
b_{2n} + \frac{b_{2n+1}}{2} + \frac{b_{2n+2}}{2^2} + \cdots
\]
Now we choose the digits $b_{n+1}$, \ldots , $b_{2n}$ in such a way that
\[
\left(b_{n+1}-1 \right) 2^{n-1} + b_{n+2} 2^{n-2} + \cdots +
b_{2n} = \ell
\]
by taking $n-2$ as the largest power of $2$ in the binary representation of $\ell$, and of course $b_{n+1}=1$. Call $n_1$ this integer $n$, and repeat. 

\begin{remark}
	A careful inspection of the previous proof shows that we have indeed
	much more freedom in the construction. We have constructed
	\[
	\alpha = 0. b_1\ldots b_{n_1}b_{n_1+1}\ldots b_{2n_1}b_{2n_1+1}\ldots
	b_{3n_1}b_{n_2+1}\ldots b_{2n_2}\ldots
	\]
	However, after $3n_1$ we could of course insert as much ``junk''
	(namely arbitrary digits) as we wish, before fixing $n_2$. As we said
	above, we are just gluing blocks of digits of $\ell$ ``sliding off''
	to the right. At each step, the number $\xi_{2^{n_j}}$ approximates
	the binary expansion of $\ell$ with higher precision.
\end{remark}

\section{Numerics} \label{SEC:num}

Motivated by the proof of Lemma \ref{LEM:cluster1} (in particular by \eqref{niven}) and with a view to studying numerically the behaviour of the sequence $\{\omega_n^+\}$ in the interval $I^+$ (see Theorem \ref{thm2}), 
we used \texttt{Wolfram MATHEMATICA 10.4.1} on a personal computer to
analyze the sequence of integers $n$ such that 
\[
|\tilde{\xi}_n| < \frac{1}{4}
\]
where $\tilde{\xi}_n$ is defined in (\ref{xi_bar:=}) and satisfies $|\tilde{\xi}_n| = n\left\| \alpha n + 1/2 \right\|$.
%
\smallskip

The code is almost trivial (here we use $\alpha = 1/\sqrt{5}$ and we
consider one milion integers):
\begin{lstlisting}
In[1]:= alpha = 1/Sqrt[5]
In[2]:= n=1
In[3]:= While[n < 1000000, If[n Abs[FractionalPart[n a] - 1/2] < 1/4, 
Print[n, " ,", N[n (FractionalPart[n a] - 1/2), 12]]]; n++]
\end{lstlisting}




This is the output corresponding to $\alpha = 1/\sqrt{5}$:

\begin{center}
	\begin{tabular}{l|l|l|l}
		$n$ & $\tilde{\xi}_n$ & $n$ & $\tilde{\xi}_n$ \\ \hline
		$1$ & $-0.05278640450004206072$ & $6119$ & $-0.05590169934418131952$ \\
		$19$ & $-0.05589202451518391926$ & $109801$ & $-0.05590169943720494638$\\
		$341$ & $-0.05590166939086236898$ & & 
	\end{tabular}
\end{center}

It seems experimentally clear that there exists a cluster point
$ 
\tilde{\xi}_\infty \approx -0.055901699.
$ 

The output corresponding to $\alpha = 1/\sqrt{3}$ is:
\begin{center}
	\begin{tabular}{l|l|l|l}
		$n$ & $\tilde{\xi}_n$ & $n$ & $\tilde{\xi}_n$ \\ \hline
		$1$ & $0.07735026918962576451$ & $2521$ & $0.07216878435841778904$ \\
		$6$ & $-0.21539030917347247767$ & $16296$ & $-0.21650635079324402136$\\
		$13$ & $0.07219549304675420205$ & $35113$ & $0.07216878365236164029$ \\
		$84$ & $-0.21650059800060562345$ & $226974$ & $-0.21650635094532167353$ \\
		$181$ & $0.07216892132967108423$ &$489061$ & $0.07216878364872207890$ \\
		$1170$ & $-0.21650632129096342623$ & & 
%
	\end{tabular}
\end{center}
In this case a different phenomenon seems to arise: there are actually
\emph{two} sequences that produce cluster points $\tilde{\xi}_{\infty,1}
\approx 0.0721687836$ and $\tilde{\xi}_{\infty,2} \approx -0.216506350$.

%

Here is the output for $\alpha = 1/(1+\sqrt{5})$:
\begin{center}
	\begin{tabular}{l|l|l|l}
		$n$ & $\tilde{\xi}_n$ & $n$ & $\tilde{\xi}_n$ \\ \hline
		$1$ & $-0.19098300562505257590$ & $6765$ & $-0.22360679677278811875$\\
		$2$ & $0.23606797749978969641$ & $10946$ & $0.22360679812323266118$ \\
		$5$ & $0.22542485937368560256$ & $28657$ & $0.22360679780443594932$\\
		$8$ & $-0.22291236000336485745$ & $46368$ & $-0.22360679772917825432$ \\
		$21$ & $-0.22350548064818597089$ & $121393$ & $-0.22360679774694418618$ \\
		$34$  & $0.22364549743922226225$ & $196418$ & $0.22360679775113815377$\\
		$89$ & $0.22361244395854631426$ & $514229$ & $0.22360679775014809233$ \\
		$144$ & $-0.22360464109021381485$ & $832040$ & $-0.22360679774991437053$\\
		$377$ & $-0.22360648309755976514$ & $2178309$ & $-0.22360679774996954476$\\
		$610$ & $0.22360691793650846339$ & $3524578$ & $0.22360679774998256962$ \\
		$1597$ & $0.22360681528495730605$ & $9227465$ & $0.22360679774997949486$\\
		$2584$ & $-0.22360679105221323713$ 
	\end{tabular}
\end{center}
Numerical evidence suggests that there are two cluster points
$\tilde{\xi}_{\infty,1} = - \tilde{\xi}_{\infty,2}$.

\medskip

We then looked up the sequences of these integers $n$ in the
\emph{Online Encyclopaedia of Integer
	Sequences}\footnote{\texttt{http://oeis.org}}. The sequence
corresponding to $\alpha = 1/\sqrt{5}$ was recognized as A049629,
namely the sequence
\[
n \mapsto \frac{F(6n+5)-F(6n+1)}{4},
\]
where $F$ is the Fibonacci sequence.
In the case $\alpha = 1/\sqrt{3}$, a first sequence was
recognized as A001570, namely numbers $n$ such that $n^2$ is a
centered hexagonal, also known as Chebyshev T-sequence with Diophantine
property. An explicit formula is known:
\[
n \mapsto \frac{(2+\sqrt{3})^{2n-1} + (2-\sqrt{3})^{2n-1}}{4}.
\]
The second sequence has been recognized as A011945, the area
of triangles with integral side lengths $m-1$, $m$, $m+1$ and integral
area.
The case $\alpha =
1/(1+\sqrt{5})$ was recognized: even integers are the so-called ``even
Fibonacci numbers'', while odd integers are defined recursively by
\[
n_0=n_1=1, \quad n_{i+1} = 4 n_{i-1} + n_{i-2}.
\]


%
On the other hand experimental numerics has not shown known patterns in correspondence to trascendental numbers.

\section{Appendix: Spectrum of $H_{\mathcal{G}}$ and bifurcation from eigenvalues}

The essential spectrum of the free Schr\"{o}dinger operator on the double bridge graph coincides with $[0,+\infty)$, see \cite{BerKu}. It also admits a countable set of embedded eigenvalues, which we now compute for completeness.

Taking into account the domain $D\left( H_{\mathcal{G}}\right)$, the
eigenvalue problem for the self-adjoint operator $H_{\mathcal{G}}$ writes
componentwise as follows: 
\begin{equation}
\left\{ 
\begin{array}{ll}
-\psi_j^{\prime \prime}=\lambda \psi_j, & \psi_j\in H^{2}(I_j),~\lambda \in \mathbb{R},\quad j=1,\ldots,4\medskip  \\ 
\psi_1(0)=\psi_2(L)=\psi_3(0),\quad  & \psi_1(L_1)=\psi_2(L_1)=\psi_4(0)\medskip  \\ 
\psi_1^{\prime}(0)-\psi_2^{\prime}(L)+\psi_3^{\prime}(0)=0,~~ & 
\psi_1^{\prime}(L_1)-\psi_2^{\prime}(L_1)-\psi_4^{\prime}(0)=0.
\end{array}
\right.   \label{eigen-pb}
\end{equation}
We set $\mu :=\sqrt{\left| \lambda \right| }$ for brevity and split the
problem into three cases, according to $\lambda <0$, $\lambda =0$ or $\lambda >0$.

If $\lambda <0$, then (\ref{eigen-pb}) is equivalent to $\psi _{j}\left(
x_{j}\right) =a_{j}e^{\mu x_{j}}+b_{j}e^{-\mu x_{j}}$ for $j=1,2$ and $\psi
_{j}\left( x_{j}\right) =b_{j}e^{-\mu x_{j}}$ for $j=3,4$ with 
\begin{equation}
\left\{ 
\begin{array}{l}
a_{1}+b_{1}=a_{2}e^{\mu L}+b_{2}e^{-\mu L}=b_{3} \\ 
a_{1}e^{\mu L_{1}}+b_{1}e^{-\mu L_{1}}=a_{2}e^{\mu L_{1}}+b_{2}e^{-\mu
	L_{1}}=b_{4} \\ 
a_{1}-b_{1}-a_{j}e^{\mu L}+b_{j}e^{-\mu L}-b_{3}=0 \\ 
a_{1}e^{\mu L_{1}}-b_{1}e^{-\mu L_{1}}-a_{2}e^{\mu L_{1}}+b_{2}e^{-\mu
	L_{1}}+b_{4}=0.
\end{array}
\right.  \label{l<0}
\end{equation}
After some computation, the determinant of this linear system turns out to
be 
\[
\Delta =2\left( e^{-\mu \left( L-L_{1}\right) }+e^{\mu \left( L-L_{1}\right)
}-3e^{\mu L}-2e^{\mu L_{1}}+e^{-\mu \left( L-2L_{1}\right) }+2\right) , 
\]
which defines a strictly decreasing function of $L$. Since $L\geq 2L_{1}$,
we have 
\begin{eqnarray*}
	\Delta &\leq &2\left( e^{-\mu L_{1}}+e^{\mu L_{1}}-3e^{2\mu L_{1}}-2e^{\mu
		L_{1}}+3\right) \\
	&=&-2e^{-\mu L_{1}}\left( 3e^{\mu L_{1}}+1\right) \left( e^{2\mu
		L_{1}}-1\right) <0
\end{eqnarray*}
and therefore (\ref{l<0}) implies $a_{1}=a_{2}=b_{1}=b_{2}=b_{3}=b_{4}=0$.

If $\lambda =0$, then (\ref{eigen-pb}) is equivalent to $\psi _{3}=\psi
_{4}=0$ and $\psi _{j}\left( x_{j}\right) =a_{j}+b_{j}x_{j}$ for $j=1,2$
with 
\[
\left\{ 
\begin{array}{l}
a_{1}=a_{2}+b_{2}L=0 \\ 
a_{1}+b_{1}L_{1}=a_{2}+b_{2}L_{1}=0 \\ 
b_{1}-b_{2}=0,
\end{array}
\right. 
\]
which readily implies $a_{1}=a_{2}=b_{1}=b_{2}=0$.

If $\lambda >0$, then (\ref{eigen-pb}) is equivalent to $\psi _{3}=\psi
_{4}=0$ and $\psi _{j}\left( x_{j}\right) =a_{j}\cos \left( \mu x_{j}\right)
+b_{j}\sin \left( \mu x_{j}\right) $ for $j=1,2$ with 
\[
\left\{ 
\begin{array}{l}
a_{1}=0 \\ 
a_{2}\cos \left( \mu L\right) +b_{2}\sin \left( \mu L\right) =0 \\ 
b_{1}\sin \left( \mu L_{1}\right) =0 \\ 
a_{2}\cos \left( \mu L_{1}\right) +b_{2}\sin \left( \mu L_{1}\right) =0 \\ 
b_{1}+a_{2}\sin \left( \mu L\right) -b_{2}\cos \left( \mu L\right) =0 \\ 
b_{1}\cos \left( \mu L_{1}\right) +a_{2}\sin \left( \mu L_{1}\right)
-b_{2}\cos \left( \mu L_{1}\right) =0.
\end{array}
\right. 
\]
If $b_{1}=0$, the second and fifth equations give $a_{2}=b_{2}=0$ and
therefore the system has only the trivial solution. If $b_{1}\neq 0$, the
third equation gives $\sin \left( \mu L_{1}\right) =0$ and thus the fourth
and last ones imply $a_{2}=0$ and $b_{1}=b_{2}$, so that the system is
equivalent to 
\[
a_{1}=a_{2}=0,\quad b_{1}=b_{2}\neq 0,\quad \sin \left( \mu L_{1}\right)
=0,\quad \sin \left( \mu L\right) =0,\quad \cos \left( \mu L\right) =1. 
\]
The last three conditions mean that there exist $k$, $h\in \mathbb{N}$ such
that $\mu L_{1}=k\pi $ and $\mu L=2h\pi $, i.e., 
\[
\frac{L_{1}}{L}=\frac{k}{2h}\quad \text{and}\quad \mu =\frac{2h\pi }{L}. 
\]

\begin{remark}   \label{q_0}
For every positive rational number $p/q$, there exists a
unique pair of coprime integers $p_0,q_0\in \mathbb{N}$ such that $p/q=p_0/(2q_0)$.
Indeed, assuming $p,q\in \mathbb{N}$ coprime, 
we can take $\left( p_0,q_0\right) =\left(2p,q\right) $ if $q\notin 2\mathbb{N}$ 
and $\left( p_0,q_0\right)=\left( p,q/2\right) $ if $q\in 2\mathbb{N}$. 
On the other hand, $p_0/(2q_0)=p_0^{\prime }/(2q_0^{\prime })$ with $p_0,q_0\in \mathbb{N}$
coprime and $p_0^{\prime },q_0^{\prime }\in \mathbb{N}$ coprime readily implies 
$p_0=p_0^{\prime }$ and $q_0=q_0^{\prime }$.
\end{remark}

As a conclusion, taking into account Remark \ref{q_0}, we obtain the following proposition.


\begin{proposition}
	If $L_{1}/L\in \mathbb{R}\setminus \mathbb{Q}$, then the operator $H_{\mathcal{G}}$ has no eigenvalues. 
	If $L_{1}/L\in \mathbb{Q}$, let $p_0$, $q_0\in \mathbb{N}$ be the unique coprime integers such that $L_{1}/L=p_0/(2q_0)$. 
	Then the eigenvalues of $H_{\mathcal{G}}$ are 
	\[
	\lambda _{n}=n^{2}\frac{4\pi ^{2}q_0^{2}}{L^{2}}
	~~\left( =n^{2}\frac{4\pi^{2}p_0^{2}}{L_{1}^{2}}\right) ,\quad n\in \mathbb{N},
	\]
	with corresponding eigenspaces 
	$E_{\lambda _{n}}=\lspan\left\{ \left( \sin\left( n\frac{2\pi q_0}{L}\,\cdot \right) ,
	\sin \left( n\frac{2\pi q_0}{L}\,\cdot \right) ,0,0\right) \right\} $.
\end{proposition}
\vskip5pt
\noindent

\noindent
We stress the fact that for $L_{1}/L\in \mathbb{R}\setminus \mathbb{Q}$  the lost eigenvalues $\lambda$ of the operator $H_{\mathcal{G}}$ become resonances, i.e. poles of the meromorphic continuation of the resolvent $(H_{\mathcal G}-k^2)^{-1}$ through the real axis to $\C^-$ (notice the change of the spectral parameter $k^2=\lambda$). The subject of these so called {\it topological resonances} is studied in several recent papers, see for example \cite{ExLip10, DP11, GSS13, LZw16, CdVTruc16}.

\vskip5pt
\noindent
As a final remark we want to show in a direct way that the compactly supported solutions of cnoidal type of the NLS equation on the double-bridge graph bifurcate, when the parameter $\omega\to \lambda_n$, from the linear eigenvectors $E_{\lambda_n}$ of the double bridge linear quantum graph discussed above.
\par\noindent
Consider the solution \eqref{cnoidalomega} having the same period $\frac{L}{nq_0}$ of the
eigenfunctions in $E_{\lambda _{n}}$, namely 
\begin{align}
u_{n,\omega }^\pm(x) &=\sqrt{\frac{2k_{n}^{2}\omega }{1-2k_{n}^{2}}}
\cn \left(\sqrt{\frac{\omega }{1-2k_{n}^{2}}}\left( x \pm \frac{T_{\omega }\left(k_{n}\right) }{4}\right) ;k_{n}\right), \\
n q_0 T_{\omega }\left(k_{n}\right) &=L, \quad
n\frac{p_0}{2}T_{\omega }\left( k_{n}\right) =L_{1}.
\label{branch}
\end{align}
Define $W\left( k\right) =\sqrt{1-2k^{2}}K\left( k\right) $ in such a way
that $T_{\omega }\left( k_{n}\right) =4W\left( k_{n}\right) /\sqrt{\omega }$
and thus 
\begin{equation}
W\left( k_{n}\right) =\frac{L\sqrt{\omega }}{4nq_0}=\frac{\pi }{2}\sqrt{\frac{\omega }{\lambda _{n}}}.
\label{W(k_n)}
\end{equation}
Notice that the function $W$ is strictly decreasing for $k\in \left( 0,1/\sqrt{2}\right) $ and satisfies 
\[
W\left( k\right) 
=\left( 1-k^{2}+o\left( k^{2}\right) \right) \left( \frac{\pi }{2}+\frac{\pi }{8}k^{2}+o\left( k^{2}\right) \right) 
=\frac{\pi }{2}-\frac{3\pi }{8}k^{2}+o\left( k^{2}\right) \text{\quad as }k\rightarrow 0.
\]
Hence
\[
\lim_{t\rightarrow \left( \pi /2\right) ^{-}}\frac{W^{-1}\left( t\right) }{\sqrt{\pi /2-t}}
=\lim_{k\rightarrow 0}\frac{k}{\sqrt{\pi /2-W\left( k\right) }}
=\sqrt{\frac{8}{3\pi }}
\]
and therefore 
\begin{equation}
W^{-1}\left( t\right) \sim \sqrt{\frac{8}{3\pi }\left( \frac{\pi }{2}-t\right) }
\text{\quad as }t\rightarrow \left( \frac{\pi }{2}\right) ^{-}.
\label{W^-1}
\end{equation}

Putting $\omega =\omega _{\varepsilon }=\lambda _{n}-\varepsilon $ in (\ref{branch}) we obtain
\[
u_{n,\omega _{\varepsilon }}^\pm(x)
=\sqrt{\frac{2k_{n}^{2}\left( \lambda_{n}-\varepsilon \right) }{1-2k_{n}^{2}}}
\cn \left( \sqrt{\frac{\lambda_{n}-\varepsilon }{1-2k_{n}^{2}}}\left( x \pm \frac{L}{4nq_0}\right) ;k_{n}\right).
\]
As $\varepsilon \rightarrow 0$, we have that 
$\frac{\pi }{2}\sqrt{\frac{\omega }{\lambda _{n}}}\rightarrow \left( \frac{\pi }{2}\right) ^{-}$ 
and therefore, using (\ref{W(k_n)}) and (\ref{W^-1}), we get that
\begin{equation*}
k_{n}
=W^{-1}\left( \frac{\pi }{2}\sqrt{\frac{\omega }{\lambda _{n}}}\right)
\sim \sqrt{\frac{4}{3}}\sqrt{1-\sqrt{1-\frac{\varepsilon }{\lambda _{n}}}} 
\sim \sqrt{\frac{2\varepsilon }{3\lambda _{n}}}
\quad \text{\quad as } \varepsilon \rightarrow 0
\end{equation*}
and thus
\[
\sqrt{\frac{2k_{n}^{2}\left( \lambda _{n}-\varepsilon \right) }{1-2k_{n}^{2}}}
\sim \sqrt{\frac{4\varepsilon }{3}}\quad \text{\quad as }\varepsilon \rightarrow
0.
\]
On the other hand, as $\varepsilon \rightarrow 0$ we have that $k_{n}\rightarrow 0$ and 
\[
\sqrt{\frac{\lambda _{n}-\varepsilon }{1-2k_{n}^{2}}}\left( x \pm \frac{L}{4nq_0}\right) 
\rightarrow \sqrt{\lambda _{n}}\left( x\pm \frac{L}{4nq_0}\right) .
\]
Hence $\cn \left( \sqrt{\frac{\lambda _{n}-\varepsilon }{1-2k_{n}^{2}}}\left( x \pm \frac{L}{4nq_0}\right) ;k_{n}\right) $ 
tends pointwise to
\[
\cn \left( \sqrt{\lambda _{n}}\left( x \pm \frac{L}{4nq_0}\right) ;0\right) 
=\cos\left( \sqrt{\lambda _{n}}\left( x \pm \frac{L}{4nq_0}\right) \right) 
=\cos \left(n\frac{2\pi q_0}{L}x \pm \frac{\pi }{2}\right) 
=\mp \sin \left( n\frac{2\pi q_0}{L} x\right) .
\]
As a conclusion, we deduce that for every $x$ and $n$ one has for $\varepsilon \rightarrow 0$
\begin{align*}
u_{n,\omega _{\varepsilon }}^\pm(x)
=\left( \sqrt{\frac{4\varepsilon}{3}}+o\left( \sqrt{\varepsilon }\right) \right) 
\left( \mp\sin \left( n\frac{2\pi q_0}{L}x\right) +o\left( 1\right) \right) 
=\mp\sqrt{\frac{4\varepsilon}{3}}\sin\left( n\frac{2\pi q_0}{L}x\right) 
+o\left( \sqrt{\varepsilon }\right).
\end{align*}


\section*{References}

\begin{thebibliography}{99}
	
	
	\bibitem{acfn-aihp} R. Adami, C. Cacciapuoti, D. Finco, D. Noja,
	{\em Constrained energy minimization and orbital stability for the NLS equation on a star graph},  
	Ann. Inst. Poincar\'e, An. Non Lin. \textbf{31} (2014), no. 6, 1289--1310. 
	
	\bibitem{AST1}  R. Adami, E. Serra, P.  Tilli, \emph{NLS ground states on graphs}, Calc. Var. and PDEs \textbf{54} (2015), no. 1, 743--761.
	
	\bibitem{AST2} R. Adami, E. Serra, P.  Tilli, \emph{Threshold phenomena and existence results for NLS ground states on metric graphs}, J. Func. An. \textbf{271} (2016), no. 1, 201--223.
	
	\bibitem{AST3} R. Adami, E. Serra, P.  Tilli, \emph{Negative Energy Ground States for the L2-Critical NLSE on Metric Graphs},
	Comm. Math. Phys. {\bf 352}, (1), 387-406 (2017)

	\bibitem{AST4} R. Adami, E. Serra, P.  Tilli, \emph {Multiple positive bound states for the subcritical NLS equation on metric graphs}, (2017), arXiv:1706.07654 
	
	
	\bibitem{BHLN01} P.A. Binding, R. Hryniv, H. Langer, B. Najman, \emph{Elliptic eigenvalue problems with eigenparameter dependent boundary conditions}, J. Differ. Equations {\bf 174}, (2001) 30-54
	
	\bibitem{BerKu} G. Berkolaiko, P. Kuchment, Introduction to Quantum Graphs, Mathematical Surveys and Monographs 186,  AMS (2013).
	
	\bibitem{CFN15}  C. Cacciapuoti, D. Finco, D. Noja, 
	\emph{Topology-induced bifurcations for the nonlinear Schr\"{o}dinger equation on the tadpole graph}, 
	Phys. Rev. E \textbf{91}, 013206 (2015).
	
	\bibitem{CFN17}  C. Cacciapuoti, D. Finco, D. Noja, \emph{Ground state and orbital stability for the NLS equation on a general starlike graph with potentials}, to appear on Nonlinearity, arXiv:1608.01506
	
	
	
	\bibitem{Cassels}  J.W.S. Cassels, 
	\emph{An introduction to diophantine approximation}, 
	Cambridge University Press (1957).
	
	\bibitem{CdVTruc16} Y. Colin De Verdi\`ere and F. Truc, \emph{Topological resonances on quantum graphs}, https://arxiv.org/abs/1604.01732 (2016)
	\bibitem{ExLip10} P. Exner and J.Lipovsky, \emph{Resonances from perturbations of quantum graphs with rationally related edges} J. Phys. A: Math. Theor., {\bf 43}, 1053 (2010)
	
	\bibitem{DP11}E. B. Davies and A. Pushnitski, \emph{Non-Weyl resonance asymptotics for quantum graphs}, Analysis\&PDE {\bf 4}, 729-756 (2011)
	
	\bibitem{GSS13} S. Gnutzmann, H. Schanz, U. Smilansky, \emph{Topological resonances in Scattering on Networks (Graphs)} PRL {\bf 110}:094101-1-094101-5 (2013)
	
	\bibitem{Komatsu97} T. Komatsu,  \emph{On Inhomogeneous Continued Fraction Expansions and Inhomogeneous Diophantine Approximation} J. Number Theory {\bf 62}, 192-212 (1997) 
	
	\bibitem{LZw16} M. Lee, M. Zworski, \emph{A Fermi golden rule for quantum graphs}, J. Math. Phys, {\bf 57}, 092101 (2016).
	
	\bibitem{MakThomp12} A. S. Makin and H. B. Thompson, \emph{On Eigenfunction Expansions for a Nonlinear Sturm-Liouville Operator with Spectral-Parameter Dependent Boundary Conditions}, Differential Equations, {\bf 48} (2) 176-188 (2012)
	
	\bibitem{Niven}  I. Niven, 
	\emph{Diophantine approximations}, 
	Interscience Tracts in Pure and Applied Mathematics No. 14, Interscience Publishers (1963).
	
	\bibitem{GW1}
	S. Gnutzmann, D. Waltner, \emph{Stationary waves on nonlinear quantum graphs. I. General framework and canonical perturbation theory},
	Phys. Rev E {\bf 93},  032204 (2016)\ 
	
	\bibitem{GW2}
	S. Gnutzmann, D. Waltner, \emph{Stationary waves on nonlinear quantum graphs. II. Application of canonical
		perturbation theory in basic graph structures},
	Phys. Rev E {\bf 94}, 062216 (2016)\ 
	
	
	\bibitem{Lawden} D.F. Lawden, Elliptic functions and applications, Springer (1989)
	
	\bibitem{[MP16]}  J. Marzuola, D. E. Pelinovsky, \emph{Ground states on the dumbbell graph}, Applied Mathematics Research Express 2016, 98--145 (2016).
	
	
	\bibitem{[NPS15]} D. Noja, D. Pelinovsky, and G. Shaikhova, \emph{Bifurcation and stability of standing waves in the nonlinear Schr\"odinger
		equation on the tadpole graph}, Nonlinearity {\bf 28} (2015), 2343--2378.
	
	\bibitem{handbook}  F.W.J. Olver, D.W. Lozier, R.F. Boisvert, C.W. Clark, 
	\emph{NIST Handbook of Mathematical Functions}, 
	Cambridge University Press, New York (2010).
	
	\bibitem{Pinner01} C.G. Pinner, \emph{More on inhomogeneous Diophantine approximation}, J.de Th\'eorie de Nombres de Bordeaux {\bf 13}, 539-557 (2001)
	
	\bibitem{PS16} D. E. Pelinovsky, G. Schneider, \emph{Bifurcations of standing localized waves on periodic graphs}, Ann. Henri Poincar\'e, {\bf 18}, 1185, (2017)
	
	\bibitem{Schmidt}  W.M. Schmidt, 
	\emph{Diophantine approximations},  Lecture Notes in Mathematics No. 785, Springer (1980).
	
	
\end{thebibliography}
\end{document}